\documentclass[10pt,twoside,reqno]{amsart}
\usepackage{amssymb, latexsym}
\usepackage[italian,english]{babel}
\usepackage[latin1]{inputenc} 

\usepackage{color}

\numberwithin{equation}{section}

\def\Xint#1{\mathchoice 
  {\XXint\displaystyle\textstyle{#1}}%
  {\XXint\textstyle\scriptstyle{#1}}%
  {\XXint\scriptstyle\scriptscriptstyle{#1}}%
  {\XXint\scriptscriptstyle\scriptscriptstyle{#1}}%
  \!\int} 
\def\XXint#1#2#3{{\setbox0=\hbox{$#1{#2#3}{\int}$} 
  \vcenter{\hbox{$#2#3$}}\kern-.5\wd0}} 
\def\-int{\Xint -}

\newcommand{\R}{\mathbb{R}}

\newcommand{\V}{\varphi}
\newcommand{\z}{\overline{z}}
\newcommand{\X}{\mathbb{X}}
\newcommand{\Y}{\mathbb{Y}}

\DeclareMathOperator{\dive}{div}

\DeclareMathOperator{\F}{\mathcal{F}}
\DeclareMathOperator{\B}{\mathcal{B}}
\DeclareMathOperator{\E}{\mathbb{E}}
\DeclareMathOperator{\A}{\mathcal{A}}

\newtheorem{prop}{Proposition}[section]
\newtheorem{lem}{Lemma}[section]
\newtheorem{thm}{Theorem}[section]
\newtheorem{defn}{Definition}[section]
\newtheorem{cor}{Corollary}[section]

\newtheorem{remark}{Remark}[section]
\newtheorem{assumption}[thm]{Assumption}

\begin{document}
\title[Partial Regularity Results...]{Partial Regularity Results \\
for Asymptotic Quasiconvex Functionals \\
with General Growth}    

\author{Teresa Isernia, Chiara Leone and  Anna Verde}
\address{Dipartimento di Matematica e Applicazioni "R. Caccioppoli"\\
         Universit\`a degli Studi di Napoli Federico II\\
         via Cinthia, 80126 Napoli, Italy}
\email{teresa.isernia@unina.it, chileone@unina.it, anverde@unina.it}

\maketitle

\begin{abstract}
We prove partial regularity for minimizers of vectorial  integrals of the Calculus of Variations, with general growth condition,  imposing quasiconvexity assumptions only in an asymptotic sense. 
\end{abstract}

\section{\bf Introduction}
\noindent
In this paper we study variational integrals of the type
\begin{equation*}
\F(u):=\int_{\Omega} f(Du) \,dx \quad \mbox{ for } u:\Omega \rightarrow \R^{N}
\end{equation*}
where $\Omega$ is an open bounded set in $\R^{n}$, $n\geq 2$, $N\geq 1$. 
Here $f:\R^{Nn}\rightarrow \R$ is a continuous function satisfying a $\V$-growth condition:
\begin{equation*}
|f(z)|\leq C(1+ \varphi(|z|) ), \quad \forall z\in \R^{Nn},
\end{equation*}
where  $C$ is a positive constant and $\V$ is a given $N$-function (see Definition \ref{defV}).

\noindent
Some examples of $N$-functions are the following:
\begin{align*}
&\V(t)= t^{p} \quad 1<p<\infty,\\
&\V(t)=t^{p}\log^{\alpha}(1+t), \, p>1 \mbox{ and } \alpha > 0.
\end{align*}
If $\varphi$ is an $N$-function satisfying the $\Delta_2$-condition (see Section \ref{defi}),  by $L^\varphi(\Omega)$ and $W^{1,\varphi}(\Omega)$  we
denote the classical Orlicz and Orlicz-Sobolev spaces, i.e.\ $u \in
L^\varphi(\Omega)$ if and only if $\displaystyle{\int_\Omega \varphi(|u|)\,dx < \infty}$ and
$u \in W^{1,\varphi}(\Omega)$ if and only if $u,  Du \in
L^\varphi(\Omega)$. The Luxembourg norm is defined as follows:
\begin{align*}
  \|u\|_{L^\varphi(\Omega)}=\inf \Big\{\lambda>0 : \int_{\Omega} \varphi
    \bigg(\frac{|u(x)|}{\lambda} \bigg)\,dx\leq 1\Big\}.
\end{align*}
With this norm $L^\varphi(\Omega)$ is a Banach space.

\noindent
If there is no misunderstanding, we will write  $\|u\|_{\varphi}$.
Moreover, we denote by $W^{1,\varphi}_0(\Omega)$ the closure of $C^{\infty}_{c}(\Omega)$ functions with respect to the norm
\begin{align*}
  \|u\|_{1,\varphi}=\|u\|_\varphi+\| Du\|_\varphi
\end{align*}
and by $W^{-1,\varphi}(\Omega)$ its dual.
\newpage

\noindent
We will consider  the following definition of a minimizer of $\mathcal{F}$.
\begin{defn}
A map $u\in W^{1, \V}(\Omega, \R^{N})$ is a $W^{1,\V}$-minimizer of $\mathcal{F}$ in $\Omega$ if 
$$
\mathcal{F}(u)\leq \mathcal{F}(u+\xi)
$$
for every $\xi \in W^{1, \V}_{0}(\Omega, \R^{N})$. 
\end{defn}

\medskip

\noindent
Let us recall the notion of quasiconvexity introduced by Morrey \cite{Morrey}:

\begin{defn}[Quasiconvexity]
A continuous function $f:\R^{Nn}\rightarrow \R$ is said to be quasiconvex if and only if
\begin{equation*}
\-int_{\B_{1}} f(z+D\xi)\, dx \geq f(z)
\end{equation*}
holds for every $z\in \R^{Nn}$ and every smooth function $\xi: \B_{1} \rightarrow \R^{N}$ with compact support in 
the unit ball $\B_{1}$ in $\R^{n}$. 
\end{defn}
\noindent Let us note that by scaling and translating, the unit ball in the definition above may be replaced by an arbitrary ball in $\R^{n}$. 
\smallskip

\noindent
The quasiconvexity was originally introduced for proving the lower semicontinuity and the existence of minimizers of variational integrals of the Calculus of Variations. In fact, assuming a power growth condition on $f$, quasiconvexity is a necessary and sufficient condition for the sequential lower semicontinuity on $W^{1,p}(\Omega,\R^N)$, $p>1$ (see \cite{AF1} and \cite{Marc}). In the regularity theory a stronger definition, the strict quasiconvexity, is needed, a notion which has nowadays become a common condition in the vectorial Calculus of Variations (see \cite{evans},\cite{AF2}, \cite{CFM}).

\noindent 
In order to treat the general growth case, we consider the notion of  strictly $W^{1,\V}$-quasiconvexity introduced in \cite{DLSV}.

\begin{defn}[Strict $W^{1,\V}$-quasiconvexity]  
A continuous function $f:\R^{Nn}\rightarrow \R$ is said to be strictly $W^{1,\varphi}$-quasiconvex if there exists a positive constant $k>0$ such that
\begin{equation*}
\-int_{\B_{1}} f(z+D\xi)\, dx \geq f(z) + k\-int_{\B_{1}} \varphi_{|z|}(|D\xi|)\, dx
\end{equation*}
for all $\xi\in C^{1}_{0}(\B_{1})$, for all $z\in \R^{Nn}$, where $\varphi_{a}(t)\sim t^{2}\varphi''(a+t)$
for $a,t\geq 0$. 
\end{defn}

\noindent
A precise definition of $\varphi_a$ is given in Section \ref{defi}. 
\smallskip

\noindent
In this paper we will exhibit an adequate notion of strict $W^{1,\V}$-quasiconvexity at infinity which we will call $W^{1,\V}$-asymptotic quasiconvexity. We will establish several characterizations of this notion and we will prove a partial regularity result, namely that minimizers are Lipschitz continuous on an open and dense subset of $\Omega$. We remark that a counterexample  of \cite{SS} shows that it is not possible to establish regularity outside a negligible set (which would be the natural thing in the vectorial regularity theory). So, our regularity result generalizes the ones given in \cite{SS} and \cite{CPSV} for integrands with a power growth condition which become strictly convex and strictly quasiconvex near infinity, respectively. 

\noindent
We also point out that in recent years a growing literature has considered the subject of asymptotic regular problems: regularity theory for integrands with a particular structure near infinity has been investigated first in \cite{CE} and subsequentely in \cite{GM}, \cite{R}, \cite{CGM}, \cite{DK}, \cite{LPV}, \cite{PV}, \cite{F}, \cite{FPV}, \cite{DSV2}, \cite{DKZ}.
\smallskip

\noindent
We deal with the problem wondering if, when you {\it localize at infinity} the natural assumptions to have regularity, this regularity breaks down or not. It is the same question faced in \cite{AG} and \cite{AF3}, where you do not require a global strict convexity or quasiconvexity assumption: all the hypotheses are localized in some point $z_0$ and you obtain that minimizers are H{\"o}lder continuous near points where the integrand function is "close" to the value $z_0$.

\noindent
In order to achieve the regularity result we have to prove an {\it excess decay estimate} (see Section $7$). In the power case the main idea is to use a blow-up argument based strongly on the homogeneity of $\V(t)=t^{p}$. Here we have to face with the lack of the homogeneity since the general growth condition. Thus one makes use of the so-called $\A$-harmonic approximation proved in \cite{DLSV} (see also \cite{S, DG, DGK, DM, DM1} for the power case). Such tool allows us to compare the solutions of our problem with the solution of the regular one in terms of the closeness of the gradient.


\section{\bf Definitions and assumptions}\label{defi}
\noindent
To simplify the presentation, the letters $c, C$ will denote generic positive constants, which may change from line to line, but does not depend on the crucial quantities. \\
For $v\in L^{1}_{loc}(\R)$ and a ball $\B_{r}(x_{0})\subset \R^{n}$ we define
$$
(v)_{\B_{r}(x_{0})}:= \-int_{\B_{r}(x_{0})} v(x)\,dx := \frac{1}{|\B_{r}(x_{0})|} \int_{\B_{r}(x_{0})} v(x)\, dx
$$
where $|\B_{r}(x_{0})|$ is the $n$-dimensional Lebesgue measure of $\B_{r}(x_{0})$. When it is clear from the context we shall omit the center as follows: $\B_{r}\equiv \B_{r}(x_{0})$. 

\subsection{$N$-functions}

The following definitions and results are standard in the context of $N$-functions (see \cite{raoren}).
\begin{defn}\label{defV}
A real function $\varphi: [0, \infty)\rightarrow [0, \infty)$ is said to be an $N$-function if  $\varphi(0)=0$ and 
there exists a right continuous nondecreasing derivative $\V'$ satisfying $\varphi'(0)=0$,
$\varphi'(t)>0$ for $t>0$ and $\displaystyle{\lim_{t\rightarrow \infty} \varphi'(t)=\infty}$. Especially $\varphi$ is convex. 
\end{defn}

\begin{defn}
We say that $\varphi$ satisfies the $\Delta_{2}$-condition (we shall write $\V\in \Delta_{2}$) if there exists a constant $c>0$ such that  
$$
\varphi(2t)\leq c\,\varphi(t) \quad \mbox{ for all } t\geq 0.
$$
We denote the smallest possible constant by $\Delta_{2}(\varphi)$. 
\end{defn}

\noindent
We shall say that two real functions $\varphi_1$ and $\varphi_2$ are {\it equivalent} and write $\varphi_1\sim\varphi_2$ if  there exist constants $c_{1}, c_{2}>0$ such that $c_{1}\varphi_1(t)\leq \varphi_2(t)\leq c_{2}\varphi_1(t)$ if $t\ge 0$. 

\noindent
Since $\varphi(t)\le\varphi(2t)$ the $\Delta_{2}$-condition implies $\varphi(2t)\sim \varphi(t)$. Moreover if $\V$ is a function satisfying the $\Delta_{2}$-condition, then $\V(t)\sim \V(at)$ uniformly in $t\geq 0$ for any fixed $a>1$.

\noindent
Let us also note that, if $\varphi$ satisfies the $\Delta_2$-condition, then any $N$-function which is equivalent to $\varphi$ satisfies this condition too.
\smallskip

\noindent

\noindent
By $(\varphi')^{-1}:[0, \infty)\rightarrow [0, \infty)$ we denote the function
\begin{equation*}
(\varphi')^{-1}(t):=\sup\{s\geq 0: \varphi'(s)\leq t\}.
\end{equation*}
If $\varphi'$ is strictly increasing, then $(\varphi')^{-1}$ is the inverse function of $\varphi'$. 

\noindent
Then $\varphi^{*}:[0, \infty) \rightarrow [0,\infty)$ with
\begin{equation*}
\varphi^{*}(t):=\int_{0}^{t} (\varphi')^{-1}(s)\,ds
\end{equation*}
is again an $N$-function and for $t>0$ it results $(\varphi^{*})'(t)=(\varphi')^{-1}(t)$. It is the complementary function of $\varphi$. Note that $\varphi^{*}(t)=\sup_{s\geq 0}(st-\varphi(s))$ and $(\varphi^{*})^{*}=\varphi$. 
Examples of such complementary pairs are
\begin{align*}
&\V(t)= \frac{t^{p}}{p} (\ln t)^{\gamma_{1}} (\ln \ln t)^{\gamma_{2}} \cdots (\ln \ln \cdots \ln t)^{\gamma_{n}}, \\
&\V^{*}(t)= \frac{t^{q}}{q} [(\ln t)^{-\gamma_{1}} (\ln \ln t)^{-\gamma_{2}} \cdots (\ln \ln \cdots \ln t)^{-\gamma_{n}}]^{q-1}
\end{align*}
with $1<p<\infty$, $\frac{1}{p}+ \frac{1}{q}=1$ and $\gamma_{i}$ are arbitrary numbers. 

\noindent
If $\V, \V^{*}$ satisfy the $\Delta_{2}$-condition we will write that $\Delta_{2}(\V, \V^{*})<\infty$. 
Assume that $\Delta_{2}(\V, \V^{*})<\infty$. Then for all $\delta>0$ there exists $c_{\delta}$ depending only on $\Delta_{2}(\varphi, \varphi^{*})$ such that
for all $s,t\geq 0$ it holds that
\begin{align*}
&t\, s\leq \delta \, \varphi(t)+c_{\delta} \, \varphi^{*}(s). 
\end{align*}
This inequality is called Young's inequality. For all $t\geq 0$
\begin{align*}
&t\leq \V^{-1}(t) (\V^{*})^{-1}(t)\leq 2t \\
&\frac{t}{2}\varphi'\Bigl(\frac{t}{2}\Bigr) \leq \varphi(t) \leq t\varphi'(t) \\
&\varphi\Bigl(\frac{\varphi^{*}(t)}{t}\Bigr) \leq \varphi^{*}(t) \leq \varphi\Bigl(\frac{2\varphi^{*}(t)}{t}\Bigr).
\end{align*}
Therefore, uniformly in $t\geq 0$,
\begin{equation}\label{new}
\varphi(t)\sim t\varphi'(t), \quad \varphi^{*}(\varphi'(t))\sim \varphi(t), 
\end{equation}
where constants depend only on $\Delta_{2}(\varphi, \varphi^{*})$.


\noindent
\begin{defn}
We say that an $N$-function $\varphi$ is of type $(p_{0}, p_{1})$ with $1\leq p_{0}\leq p_{1}<\infty$ if
\begin{equation}\label{max}
\varphi(st)\leq C \max\{s^{p_{0}}, s^{p_{1}}\}\varphi(t) \quad \forall s,t\geq 0.
\end{equation} 
\end{defn}

\noindent
The following Lemma can be found in \cite{DLSV} (see Lemma 5).

\begin{lem}\label{p0p1}
Let $\V$ be an $N$-function with $\V\in\Delta_2$ together with its conjugate. Then $\V$ is of type $(p_{0}, p_{1})$ with $1< p_{0}<p_{1}<\infty$ where $p_0$ and $p_1$ and the constant $C$ depend only on $\Delta_2(\V,\V^*)$.
\end{lem}

\noindent
Throughout the paper we will assume that $\V$ satisfies the following assumption.

\begin{assumption}\label{ass:phi}
Let $\varphi$ be an $N$-function such that $\V$ is $C^1([0,+\infty))$ and $C^2(0,+\infty)$. Further assume that
\begin{equation}\label{V'}
\V'(t)\sim t\V''(t).
\end{equation}
\end{assumption}

\noindent
We remark that under this assumption $\Delta_2(\V,\V^*)<\infty$ will be automatically satisfied, where $\Delta_2(\V,\V^*)$ depends only on the characteristics of $\V$.

\noindent
We consider a family of $N$-functions $\{\varphi_{a}\}_{a\geq 0}$ setting, for  $t\geq 0$,
\begin{equation*}\label{phia}
\varphi_{a}(t):=\int_{0}^{t} \varphi'_{a}(s) \, ds \quad \mbox{ with } \quad 
\varphi'_{a}(t):= \varphi'(a+t) \frac{t}{a+t}. 
\end{equation*}

\noindent
The following lemma can be found in \cite{DE} (see Lemma 23 and Lemma 26).

\begin{lem}\label{phia}
Let $\V$ be an $N$-function with $\V\in\Delta_2$ together with its conjugate. Then for all  $a\geq 0$ the function $\varphi_a$ is an $N$-function and $\{\varphi_{a}\}_{a\geq 0}$ and $\{(\varphi_{a})^{*}\}_{a\geq 0}\sim\{\varphi^*_{\varphi'(a)}\}_{a\geq 0}$ satisfy the $\Delta_{2}$ condition
uniformly in $a\geq 0$.
\end{lem}

\noindent
Let us observe that by the previous lemma $\varphi_{a}(t)\sim t\varphi'_{a}(t)$. Moreover, for $t\geq a$ we have $\varphi_{a}(t)\sim \varphi(t)$
and for $t\leq a$ we have $\varphi_{a}(t)\sim t^{2} \varphi''(a)$. This implies that $\varphi_{a}(st)\leq c s^{2} \varphi_{a}(t)$
for all $s\in [0,1]$, $a\geq 0$ and $t\in [0,a]$. 

\noindent
For given $\V$ we define the associated $N$-function $\psi$ by 
\begin{equation*}
\psi'(t)=\sqrt{t\varphi'(t)}.
\end{equation*}
Note that 
\begin{align*}
\psi''(t)&= \frac{1}{2} \Bigl(\frac{\V''(t)}{\V'(t)}t +1\Bigr)\sqrt{\frac{\V'(t)}{t}} \\
&=\frac{1}{2} \Bigl(\frac{\V''(t)}{\V'(t)}t +1\Bigr) \frac{\psi '(t)}{t}. 
\end{align*}
It is shown in \cite{DE} (see Lemma 25) that if $\V$ satisfies Assumption \ref{ass:phi} then also $\varphi^{*}$, $\psi$ and $\psi^{*}$ satisfy Assumption \ref{ass:phi} and $\psi''(t)\sim \sqrt{\V''(t)}$. 

\noindent
We define $A,V:\R^{Nn}\rightarrow \R^{Nn}$ in the following way:
\begin{equation}\begin{split}\label{AVlambda1}
&A(z) = D\Phi(z) \\
&V(z) = D\Psi(z),
\end{split}\end{equation}
where $\Phi(z):=\V(|z|)$ and $\Psi(z):=\psi(|z|)$.
About the functions $A$ and $V$, the following three lemmas can be found in \cite{DE} (see Lemma 21, Lemma 24, and Lemma 3, respectively).
\begin{lem}\label{thmA1}
Let $\V$ satisfying Assumption \ref{ass:phi}, then  $A(z)=\V'(|z|)\frac{z}{|z|}$ for $z\neq 0$, $A(0)=0$ and $A$ satisfies 
\begin{align*}
&|A(z_{1})- A(z_{2})| \leq c \V''(|z_{1}|+|z_{2}|)|z_{1}-z_{2}|  \\
&(A(z_{1})- A(z_{2}), z_{1}-z_{2}) \geq C \V''(|z_{1}|+|z_{2}|)|z_{1}-z_{2}|^{2},
\end{align*}
for $z_{1},z_{2} \in \R^{Nn}$. 
\end{lem}

\noindent
The same conclusions of Lemma \ref{thmA1} holds with $A$ and $\V$ replaced by $V$ and $\psi$. 

\begin{lem}\label{thmA2}
Let $\V$ satisfy Assumption \ref{ass:phi}. Then, uniformly in $z_{1},z_{2}\in \R^{n}$, $|z_{1}|+|z_{2}|>0$
\begin{align*}
&\V''(|z_{1}|+|z_{2}|)|z_{1}-z_{2}| \sim \V'_{|z_{1}|}(|z_{1}-z_{2}|) , \\
&\V''(|z_{1}|+|z_{2}|)|z_{1}-z_{2}|^{2} \sim \V_{|z_{1}|}(|z_{1}-z_{2}|). 
\end{align*}
\end{lem}

\noindent

\begin{lem}\label{lem4}
Let $\V$ satisfy Assumption \ref{ass:phi} and let $A$ and $V$ be defined by (\ref{AVlambda1}). 
Then, uniformly in $z_{1}, z_{2} \in \R^{Nn}$, 
\begin{align*}
(A(z_{1})- A(z_{2}), z_{1}-z_{2}) &\sim |V(z_{1})- V(z_{2})|^{2}\sim \V_{|z_{1}|}(|z_{1}-z_{2}|),
\end{align*}
and 
\begin{align*}
|A(z_{1})- A(z_{2})| &\sim \V'_{|z_{1}|}(|z_{1}-z_{2}|). 
\end{align*}
Moreover 
\begin{align*}
&(A(z_{1}), z_{1}) \sim |V(z_{1})|^{2} \sim \V(|z_{1}|), \\
&|A(z_{1})|\sim \V'(|z_{1}|),
\end{align*}
uniformly in $z_{1}\in \R^{Nn}$.
\end{lem}

\subsection{Asymptotic $W^{1,\V}$-quasiconvexity}

Before introducing the notion of asymptotic $W^{1,\V}$-quasiconvexity, let us consider a uniform version of the strict $W^{1,\V}$-quasiconvexity.

\begin{defn}[Uniform strict $W^{1,\V}$-quasiconvexity]
A continuous function $f:\R^{Nn}\to\R$ is said to be uniformly strictly $W^{1,\varphi}$-quasiconvex  if there exists a positive constant $k>0$ such that
\begin{equation}\label{usqc}
\-int_{\B_{1}} f(z+D\xi)\, dx \geq f(z) + k\-int_{\B_{1}} \varphi_{1+|z|}(|D\xi|) \, dx
\end{equation}
for all $\xi\in C^{1}_{c}(\B_{1})$, for all $z\in \R^{Nn}$, where $\varphi_{a}(t)\sim t^{2}\varphi''(a+t)$ for $a,t\geq 0$. 
\end{defn}

\noindent

\begin{defn}[Asymptotic $W^{1,\varphi}$-quasiconvexity]
A function $f:\R^{Nn}\to\R$ is asymptotically $W^{1, \V}$-quasiconvex if there exist a positive constant $M$ and a 
uniformly strictly $W^{1,\varphi}$-quasiconvex function $g$ such that 
\begin{equation*}
f(z)=g(z) \, \mbox{ for } |z|>M.
\end{equation*}
\end{defn}
\smallskip

\noindent
Considering an $N$-function satisfying Assumption \ref{ass:phi}, we will work with the following set of hypotheses.

\subsection{\bf Assumptions} Let $f:\R^{Nn}\to\R$ be such that

\begin{itemize}
\item [$(\mathcal{H}1)$] $f\in C^{1}(\R^{Nn})\cap C^{2}(\R^{Nn}\setminus\{0\})$; 
\medskip
\item [$(\mathcal{H}2)$] $\forall z\in \R^{Nn}, \ |f(z)|\leq C(1+ \varphi(|z|))$;
\medskip
\item [$(\mathcal{H}3)$] $f$ is asymptotically $W^{1,\V}$-quasiconvex;
\medskip
\item [$(\mathcal{H}4)$]$\forall z\in \R^{Nn}\setminus\{0\},  \ |D^{2}f(z)|\leq C\, \varphi''(|z|)$;
\medskip
\item [$(\mathcal{H}5)$] $\forall z_{1}, z_{2} \in \R^{Nn}$  such that $|z_{1}|\leq \frac{1}{2} |z_{2}|$ it holds
$$
|D^{2}f(z_{2})- D^{2}f(z_{2}+z_{1})|\leq C\, \varphi''(|z_{2}|)|z_{2}|^{-\beta}|z_{1}|^{\beta}.
$$
\end{itemize}
\medskip

\begin{remark}
Due to hypothesis $(\mathcal{H}2)$, $\mathcal{F}$ is well defined on the Sobolev-Orlicz space $W^{1, \V}(\Omega, \R^{N})$. 

\noindent
Let us also observe that Assumption $(\mathcal{H}5)$, that is a H\"{o}lder continuity of $D^{2}f$ away from zero, has been used to show everywhere regularity of radial functionals with $\V$-growth (see \cite{DSV1}). We will use it in Lemma \ref{thmA6} below. 
\end{remark}


\section{\bf Technical Lemmas}
\noindent
For $z_{1}, z_{2} \in \R^{Nn}$, $\theta \in [0,1]$ we define $z_{\theta}= z_{1} + \theta (z_{2}-z_{1})$.
The following fact can be found in \cite{AF} (see Lemma 2.1).

\begin{lem}\label{lem1}
Let $\beta >-1$, then  uniformly in $z_{1}, z_{2} \in \R^{Nn}$ with $|z_{1}|+|z_{2}|>0$, it holds:
\begin{equation*}
\int_{0}^{1} |z_{\theta}|^{\beta} \, d\theta \sim (|z_{1}|+|z_{2}|)^{\beta}. 
\end{equation*}
\end{lem}

\noindent
Next result is a slight generalization of Lemma $20$ in \cite{DE}.
\begin{lem}\label{lem2}
Let $\varphi$ be an $N$-function with $\Delta_{2}(\{\varphi,\varphi*\})<\infty$; then, uniformly in $z_{1}, z_{2} \in \R^{Nn}$  with $|z_{1}|+|z_{2}|>0$, and in $\mu\geq 0$, it holds
\begin{equation*}
 \frac{\varphi'(\mu+|z_{1}|+|z_{2}|)}{\mu +|z_{1}|+|z_{2}|} \sim \int_{0}^{1} \frac{\varphi'(\mu+|z_{\theta}|)}{\mu+|z_{\theta}|} d\theta.
\end{equation*} 
\end{lem}

\noindent
From the previous lemmas we derive the following one.
\begin{lem}\label{lemtV''}
Let $\V$ be an $N$-function satisfying Assumption \ref{ass:phi}. Then, uniformly in $z_{1}, z_{2}\in \R^{Nn}$  with $|z_{1}|+|z_{2}|>0$, and  in $\mu\geq 0$, it holds 
\begin{align*}
\int_{0}^{1}\int_{0}^{1} t\V''(\mu+|z_{1}+st z_{2}|) \, ds\,dt \sim \V''(\mu+|z_{1}|+|z_{2}|).
\end{align*}
\end{lem}

\begin{proof}
Using $\V'(t)\sim t\V''(t)$, applying twice Lemma \ref{lem2}, and taking into account that $\mu+|z_{1}|+|z_{1}+z_{2}|\sim \mu+|z_{1}|+|z_{2}|$ and $\V'(2t)\sim \V'(t)$, we obtain
\begin{align*}
\int_{0}^{1}\int_{0}^{1} t\V''(\mu+|z_{1}+st z_{2}|) \, ds\,dt &\leq 
c \int_{0}^{1}\int_{0}^{1} t \frac{\V'(\mu+|z_{1}+st z_{2}|)}{\mu+|z_{1}+st z_{2}|} \, ds dt\\
&\leq c \frac{\V'(\mu+|z_{1}|+|z_{1}|+|z_{1}+z_{2}|)}{\mu+|z_{1}|+|z_{1}|+|z_{1}+z_{2}|}\\
&\leq c \frac{\V'(\mu+|z_{1}|+|z_{2}|)}{\mu+|z_{1}|+|z_{2}|}\\
&\leq c \V''(\mu+|z_{1}|+ |z_{2}|). 
\end{align*}
Similarly, for the other inequality, we have
\begin{align*}
\int_{0}^{1} \int_{0}^{1} t\V''(\mu+|z_{1}+st z_{2}|)\, ds\,dt &\geq 
c\int_{0}^{1} \int_{0}^{1} t\frac{\V'(\mu+|z_{1} +st z_{2}|)}{\mu+|z_{1} +st z_{2}|} \, ds\,dt \\
&\geq c \int_{0}^{1} t \frac{\V'(\mu+|z_{1}|+ |z_{1}+t z_{2}|)}{\mu+|z_{1}| +|z_{1} + t z_{2}|} \, dt \\
&\ge \frac{c}{(\mu+|z_1|+|z_2|)^2}\int_0^1  \varphi(\mu+|z_1|+|z_1+t z_2|)\,t \, dt,
\end{align*}
where, in the last line, we used that $\varphi(t)\sim t\varphi'(t)$.
Due to the Jensen inequality, we go ahead and we obtain
\begin{align*}
\int_{0}^{1} \int_{0}^{1} t\V''(\mu+|z_{1}+st z_{2}|)\, ds\,dt &\ge \frac{c}{(\mu+|z_1|+|z_2|)^2}\,\varphi\left(\int_0^1(\mu+|z_1|+|z_1+t z_2|)\,t\,dt\right)\\
&\ge \frac{c}{(\mu+|z_1|+|z_2|)^2}\,\varphi(\mu+|z_1|+|z_2|)\\
&\ge c\frac{\varphi'(\mu+|z_1|+|z_2|)}{\mu+|z_1|+|z_2|}\ge c\,\varphi''(\mu+|z_1|+|z_2|),
\end{align*}
thanks also to the equivalence between $\V(2t)$ and $\V(t)$, $\V(t)$ and $t\V'(t)$, and $\V'(t)$ and $t\V''(t)$.
\end{proof}

\begin{remark}\label{REM}
From the previous lemma we easily deduce that
$$
\int_{0}^{1}\int_{0}^{1} t\V''(\sqrt{1+|z_{1}+st z_{2}|^2}) \, ds\,dt \sim  \V''(1+|z_{1}|+|z_{2}|),
$$
since $\V'(t)\sim t\V''(t)$, $\V'$ is increasing and $\V'(2t)\sim \V'(t)$. 
\end{remark}

\noindent
The following version of the Sobolev-Poincar\'e inequality can be found in \cite{DE} (Theorem 7):
\begin{thm}\label{SP}
Let $\V$ be an $N$-function with $\Delta_{2}(\V, \V^{*})<\infty$. Then there exist $\alpha\in (0,1)$
and $k>0$ such that, if $\B\subset \R^{n}$ is a ball of radius $R$ and $u\in W^{1, \V}(\B, \R^{N})$, then
\begin{equation*}
\-int_{\B} \V\Bigl	(\frac{|u-(u)_{\B}|}{R}	\Bigr) dx \leq k \Bigl(\-int_{\B} \V^{\alpha}(|Du|) \, dx\Bigr)^{\frac{1}{\alpha}}.
\end{equation*}
\end{thm}

\noindent
The following two lemmas will be useful later.
 
\begin{lem}\label{cambio}
Let $\V$ satisfy Assumption \ref{ass:phi} and $p_{0}, p_{1}$ be as in Lemma \ref{p0p1}. Then for each $\eta\in (0,1]$ it holds
\begin{align*}
&\V_{|a|}(t) \leq C \eta^{1-{\bar {p}}'} \V_{|b|}(t) + \eta |V(a) - V(b)|^{2}, \\
&(\V_{|a|})^{*}(t) \leq C \eta^{1-{\bar {q}}} (\V_{|b|})^{*}(t)+ \eta  |V(a) - V(b)|^{2}
\end{align*}
for all $a,b \in \R^{n}$, $t\geq 0$ and $\bar {p}=\min\{p_0,2\}, \bar{q}=\max\{p_1,2\}$. The constants depend only on the characteristics of $\V$. 
\end{lem}

\noindent
For the proof see Lemma 2.5 in \cite{DKS}.
\medskip

\noindent

\begin{lem}\label{keylem}
Let $\V$ be an $N$-function satisfying Assumption \ref{ass:phi} and let us consider the function $z\in\R^{Nn}\mapsto\varphi(\sqrt{1+|z|^{2}} \,)$. Then, uniformly in $y,z\in \R^{Nn}$ it holds
$$
(D^{2} \V(\sqrt{1+|z+y|^{2}} \,)y,y)\sim \V''(\sqrt{1+|z+y|^{2}}\, )|y|^{2}.
$$
\end{lem}

\begin{proof}
We can see that
$$
D\V(\sqrt{1+|z+y|^{2}} \,)= \V'(\sqrt{1+|z+y|^{2}} \,) \frac{z+y}{\sqrt{1+|z+y|^{2}}},
$$
and
\begin{align*}
D^{2} \V(\sqrt{1+|z+y|^{2}}\,) &= \V''(\sqrt{1+|z+y|^{2}}\,) \frac{z+y}{\sqrt{1+|z+y|^{2}}} \otimes \frac{z+y}{\sqrt{1+|z+y|^{2}}}\\
&+\frac{\V'(\sqrt{1+|z+y|^{2}}\, )}{\sqrt{1+|z+y|^{2}}} \Bigl[\mathbb{I} - \frac{z+y}{\sqrt{1+|z+y|^{2}}} \otimes \frac{z+y}{\sqrt{1+|z+y|^{2}}}\Bigr],
\end{align*}
where $\mathbb{I}\in \R^{Nn}$ is the identity matrix.
Therefore
\begin{align*}
(D^{2} \V(\sqrt{1+|z+y|^{2}} \,)y,y)&= \V''(\sqrt{1+|z+y|^{2}}\, ) \frac{|(z+y,y)|^{2}}{1+|z+y|^{2}} \\
&+ \frac{\V'(\sqrt{1+|z+y|^{2}}\,)}{\sqrt{1+|z+y|^{2}}} \Bigl[|y|^{2} - \frac{|(z+y,y)|^{2}}{1+|z+y|^{2}} \Bigr].
\end{align*}

\noindent
Using Assumption \ref{ass:phi} and the fact that $ \displaystyle{\frac{|(z+y,y)|^{2}}{1+|z+y|^{2}} \leq |y|^{2}}$ we deduce
\begin{align*}
(D^{2} \V(\sqrt{1+|z+y|^{2}} \,)y,y)&\leq \V''(\sqrt{1+|z+y|^{2}}\, ) \frac{|(z+y,y)|^{2}}{1+|z+y|^{2}} \\
&+ C \V''(\sqrt{1+|z+y|^{2}}\, ) \Bigl[|y|^{2} - \frac{|(z+y,y)|^{2}}{1+|z+y|^{2}} \Bigr]\\
&\leq C \V''(\sqrt{1+|z+y|^{2}}\,) |y|^{2}.
\end{align*}
Similarly, 
\begin{align*}
&(D^{2} \V(\sqrt{1+|z+y|^{2}} \,)y,y)\geq \\
&\geq \V''(\sqrt{1+|z+y|^{2}}\, ) \frac{|(z+y,y)|^{2}}{1+|z+y|^{2}} 
+ C \V''(\sqrt{1+|z+y|^{2}}\, ) \Bigl[|y|^{2} - \frac{|(z+y,y)|^{2}}{1+|z+y|^{2}} \Bigr]\\
&= C \V''(\sqrt{1+|z+y|^{2}}\,) |y|^{2} + (1-C) \V''(\sqrt{1+|z+y|^{2}}\,) \frac{|(z+y,y)|^{2}}{1+|z+y|^{2}}\\
&\geq C \V''(\sqrt{1+|z+y|^{2}}\,) |y|^{2}. 
\end{align*}
\end{proof}


\section{\bf Characterization of asymptotic $W^{1,\V}$-quasiconvexity}

\noindent
In this section we will establish some characterizations of asymptotic $W^{1, \V}$-quasiconvexity. 
\begin{thm} \label{thm1}
Each of the following assertions is equivalent to the asymptotic $W^{1, \V}$-quasiconvexity of a function $f: \R^{Nn}\rightarrow \R$: 
\begin{itemize}
\item [(i)] If $f$ is $C^{2}$ outside a large ball there exists a uniformly strictly $W^{1, \varphi}$-quasiconvex function $g$ which is $C^{2}$ outside a large ball with 
\begin{equation}\label{limite}
\lim_{|z|\rightarrow \infty} \frac{|D^{2}f(z)- D^{2}g(z)|}{\varphi''(|z|)}=0 . 
\end{equation}
\end{itemize}

\begin{itemize}
\item [(ii)] If $f$ is locally bounded from below, then there exist a positive constant $M$ and a uniformly strictly 
$W^{1, \varphi}$-quasiconvex function $g$ such that 
$$
f(z)=g(z) \, \mbox{ for } |z|>M
$$
and 
$$
g\leq f \, \mbox{ on } \R^{Nn}. 
$$
\end{itemize}

\begin{itemize}
\item [(iii)] If $f$ is locally bounded from above, then there exist a positive constant $M$ and a uniformly strictly 
$W^{1, \varphi}$-quasiconvex function $g$ such that 
$$
f(z)=g(z) \, \mbox{ for } |z|>M
$$
and 
$$
g\geq f \, \mbox{ on } \R^{Nn}. 
$$
\end{itemize}

\begin{itemize}
\item [(iv)] If $f$ satisfies
 $(\mathcal{H}2)$ there exist positive constants $M,k,L$ such that
\begin{equation}\label{iv1}
\-int_{\B_{1}} f(z+D\xi) \, dx \geq f(z) +k \-int_{\B_{1}} \V_{|z|}(|D\xi|)\, dx 
\end{equation}
for $|z|>M$ and $\xi \in C^{\infty}_{c}(\B_{1}, \R^{N})$, and
\begin{equation}\label{iv2}
|f(z_{2})-f(z_{1})|\leq L|z_{1}-z_{2}| \V'(1+|z_{1}|+|z_{2}|) 
\end{equation}
for all $|z_{1}|,|z_{2}|>M$. 
\end{itemize}

\end{thm}

\begin{proof} The proof stands on four steps.

{\it Step 1}: We want to prove that $f$ asymptotically $W^{1, \V}$-quasiconvex is equivalent to (i). 
Let us show that (i) implies the asymptotic $W^{1,\V}$-quasiconvexity of $f$, the other implication being evidently true.


\noindent Let $g$ be as in (i). 
We may assume that $f,g$ are $C^{2}(\R^{Nn}\setminus \overline{\B}_{\frac{1}{2}})$ 
and taking $h=f-g$ we have that $h\in C^{2}(\R^{Nn}\setminus \overline{\B}_{\frac{1}{2}})$. In particular, by (\ref{limite}) it holds
\begin{equation}\label{limiteh}
\lim_{|z|\rightarrow \infty} \frac{|D^{2}h(z)|}{\varphi''(|z|)}=0. 
\end{equation}
Our aim is to prove that 
\begin{align}
\lim_{|z|\rightarrow \infty} \frac{|D h(z)|}{\varphi'(|z|)}=0, \label{limiteh'}
\end{align}
and 
\begin{align}
\lim_{|z|\rightarrow \infty} \frac{|h(z)|}{\varphi(|z|)}=0 \label{limiteh''}.
\end{align}

\noindent
Let us consider $|z|>1$. Take $\displaystyle{\z:=\frac{z}{|z|}}$, then $|\overline{z}|=1$ and
\begin{align*}
\frac{|Dh(z)|}{\varphi'(|z|)}&\leq \frac{1}{\varphi'(|z|)} \Bigl[ \int_{0}^{1} |D^{2}h(\overline{z}+t(z-\overline{z}))|
|z-\overline{z}| \, dt + |Dh(\overline{z})|\Bigr] \\
& = \int_{0}^{\frac{1}{\sqrt{|z|}}} \frac{|D^{2}h(\overline{z}+t(z-\overline{z}))|}{\V''(|\z+t(z-\z)|)} 
\frac{\V''(|\z+t(z-\z)|)}{\V'(|z|)}|z-\z| \, dt\\
&+ \int_{\frac{1}{\sqrt{|z|}}}^{1} \frac{|D^{2}h(\overline{z}+t(z-\overline{z}))|}{\V''(|\z+t(z-\z)|)} 
\frac{\V''(|\z+t(z-\z)|)}{\V'(|z|)}|z-\z| \, dt\\
&+ \frac{|Dh(\z)|}{\V'(|z|)} \\
&=\mathcal{I}+\mathcal{II}+\mathcal{III}. 
\end{align*}

\noindent
Estimate for $\mathcal{I}$:
\begin{align*}
\mathcal{I} &\leq \sup_{|y|>1} \frac{|D^{2} h(y)|}{\V''(|y|)} \int_{0}^{\frac{1}{\sqrt{|z|}}} \frac{\V''(|\z+t(z-\z)|)}{{\V'(|z|)}}|z-\z| \, dt\\
&\leq \sup_{|y|>1} \frac{|D^{2} h(y)|}{\V''(|y|)} \frac{1}{\V'(|z|)} 
\int_{0}^{\frac{1}{\sqrt{|z|}}} \V''(1+t(|z|-1))(|z|-1) \, dt\\
&= \sup_{|y|>1} \frac{|D^{2} h(y)|}{\V''(|y|)} \frac{1}{\V'(|z|)} \Bigl[\V'(1+t(|z|-1))\Bigr]_{0}^{\frac{1}{\sqrt{|z|}}} \\
&\leq \sup_{|y|>1} \frac{|D^{2} h(y)|}{\V''(|y|)} \frac{1}{\V'(|z|)} \V'\Bigl(1+\frac{|z|-1}{\sqrt{|z|}}\Bigr).
\end{align*}
\noindent
Taking into account that
\begin{align*}
\frac{\V'\Bigl(1+\frac{|z|-1}{\sqrt{|z|}}\Bigr)}{\V'(|z|)}&\leq \frac{\V'(1+\sqrt{|z|})}{\V'(|z|)}\\
&\leq c \frac{\V(1+\sqrt{|z|})}{1+\sqrt{|z|}}\frac{|z|}{\V(|z|)}\\
&\leq c \frac{\V(1+\sqrt{|z|})}{\V(|z|)} \sqrt{|z|}\\
&\leq c \frac{\V(\sqrt{|z|})}{\V(|z|)} \sqrt{|z|}
\end{align*}
and using Lemma \ref{p0p1} we can find $p_{0}>1$ and $C>0$ such that
$$
\V(\sqrt{|z|})= \V \Bigl(\frac{|z|}{\sqrt{|z|}}\Bigr) \leq C \Bigl(\frac{1}{\sqrt{|z|}}\Bigr)^{p_{0}} \V(|z|).
$$
Then we obtain
\begin{align}\label{q}
\frac{\V'\Bigl(1+\frac{|z|-1}{\sqrt{|z|}}\Bigr)}{\V'(|z|)}&\leq C \frac{\V(\sqrt{|z|})}{\V(|z|)} \sqrt{|z|} 
\leq C \frac{\sqrt{|z|}}{(\sqrt{|z|})^{p_{0}}} \rightarrow 0 \mbox{ as } |z|\rightarrow \infty. 
\end{align}
At this point, using (\ref{limiteh}) and (\ref{q}), we can conclude that $\mathcal{I}\rightarrow 0$ as $|z|\rightarrow +\infty$.


\noindent
Now we estimate $\mathcal{II}$:
\begin{align*}
\mathcal{II}&\leq \sup_{|y|>\sqrt{|z|}} \frac{|D^{2} h(y)|}{\V''(|y|)} \int_{\frac{1}{\sqrt{|z|}}}^{1}\frac{\V''(|\z+t(z-\z)|)}{{\V'(|z|)}}|z-\z| \, dt\\
&\leq \sup_{|y|>\sqrt{|z|}} \frac{|D^{2} h(y)|}{\V''(|y|)} \frac{1}{\V'(|z|)} \int_{\frac{1}{\sqrt{|z|}}}^{1} \V''(1+t(|z|-1))(|z|-1)\,dt\\
&\leq \sup_{|y|>\sqrt{|z|}} \frac{|D^{2} h(y)|}{\V''(|y|)} \frac{1}{\V'(|z|)} \Bigl[\V'(1+t(|z|-1))\Bigr]_{\frac{1}{\sqrt{|z|}}}^{1} \\
&= \sup_{|y|>\sqrt{|z|}} \frac{|D^{2} h(y)|}{\V''(|y|)} \frac{1}{\V'(|z|)} \Bigl[\V'(|z|)- \V'\Bigl(1+\frac{|z|-1}{\sqrt{|z|}}	\Bigr)	\Bigr]\\
&=\sup_{|y|>\sqrt{|z|}} \frac{|D^{2} h(y)|}{\V''(|y|)} \Bigl[1-\frac{\V'\Bigl(1+\frac{|z|-1}{\sqrt{|z|}}\Bigr)}{\V'(|z|)}\Bigr]
\rightarrow 0 \mbox{ as } |z|\rightarrow \infty
\end{align*}
where we used (\ref{limiteh}) and (\ref{q}) to conclude. 

\noindent
Finally 
$$
\mathcal{III}\leq \frac{1}{\V'(|z|)}\max_{|y|=1} |Dh(y)|\rightarrow 0 \mbox{ as } |z|\rightarrow \infty. 
$$

\noindent
Analogously we also obtain (\ref{limiteh''}). 
We can see that if $\frac{|D^{i} h(z)|}{\V^{(i)}(|z|)}\rightarrow 0$ as $|z|\rightarrow \infty$ for $i=0,1,2$, then 
$\frac{|D^{i} h(z)|}{\V^{(i)}(1+|z|)}\rightarrow 0$ as $|z|\rightarrow \infty$.

\noindent
Taking into account (\ref{limiteh}),(\ref{limiteh'}) and (\ref{limiteh''}), fixed $\nu>0$, that we will choose later, 
there exists $M>>1$ such that if $|z|>M$ then
\begin{align*}
&|D^{2}h(z)|\leq \nu \varphi''(1+|z|) \\
&|Dh(z)|\leq \nu \varphi'(1+|z|) \\
&|h(z)|\leq \nu \varphi(1+|z|).
\end{align*}

\noindent
Let us consider a cut-off function $\eta$ defined by
\begin{equation*}
\left\{
\begin{array}{ll}
0\leq \eta \leq 1 &\mbox{ if } 1<|x|\leq 2 \\
\eta = 1 &\mbox{ if } |x|>2 \\
\eta = 0 &\mbox{ if } |x|\leq 1.
\end{array}
\right.
\end{equation*}

\noindent
Set 
$$
\alpha:=\max \Bigl \{ \sup_{\R^{Nn}}|D\eta|, \, \sup_{\R^{Nn}}|D^{2}\eta| \Bigr \}
$$ 
and let us consider $\eta_{M}(z)= \eta(\frac{z}{M+1})$. Then we have 
$$
|D\eta_{M}|\leq \frac{\alpha}{M+1} \quad \mbox{ and } \quad |D^{2}\eta_{M}|\leq \frac{\alpha}{(M+1)^{2}}.  
$$

\noindent
Let $\Phi:=\eta_{M} h$; then for $M\leq |z|\leq 2M$ we have 
\begin{equation}
|D^{2}\Phi(z)|\leq |D^{2}\eta_{M}(z)||h(z)| + 2|D\eta_{M}(z)||Dh(z)| + |\eta_{M}(z)||D^{2}h(z)|.\nonumber
\end{equation}
Taking into account the previous estimates, (\ref{new}), (\ref{V'}) and $M\leq|z|\leq 2M$, we have
\begin{align*}
|D^{2}\Phi(z)|&\leq \frac{\nu \alpha}{(M+1)^{2}}\varphi(1+|z|) + \frac{2\nu \alpha}{(M+1)}\varphi'(1+|z|) +\nu\varphi''(1+|z|) \\
&\leq \Bigl[\frac{\nu \alpha}{(M+1)^{2}}(1+|z|)^{2} + \frac{2 \nu \alpha}{(M+1)}(1+|z|)+ \nu\Bigr] \V''(1+|z|) 
=\lambda \nu \V''(1+|z|).
\end{align*}
In particular we can conclude that 
\begin{equation}\label{k}
|D^{2}\Phi(z)|\leq \lambda \nu \varphi''(1+|z|) \quad \forall z\in \R^{Nn}.
\end{equation}
Let $\xi \in C^{\infty}_{c}(\B_{1})$; we can write
\begin{equation*}
\Phi(z+D\xi)=\Phi(z) + (D\Phi(z), D\xi) + \int_{0}^{1} (1-t)(D^{2}\Phi(z+tD\xi)D\xi, D\xi) \,dt. 
\end{equation*}
Integrating over $\B_{1}$ and by using (\ref{k}), (\ref{V'}), Lemma \ref{lem2}, the fact 
\begin{equation}\label{simbella}
1+|z|+|D\xi+z|\sim 1+|z|+|D\xi|
\end{equation}
and $\V_{a}(t)\sim \V''(a+t)t^{2}$ we get
\begin{equation}\begin{split}\label{Phi}
\-int_{\B_{1}} \Phi(z+D\xi) \,dx &= \-int_{\B_{1}} \Phi(z) \,dx + \-int_{\B_{1}} (D\Phi(z), D\xi) \, dx  \\
&+\-int_{\B_{1}} \int_{0}^{1} (1-t)(D^{2}\Phi(z+tD\xi))D\xi, D\xi) \,dt dx  \\
&\geq  \Phi(z) -\-int_{\B_{1}} \int_{0}^{1} (1-t)|D^{2}\Phi(z+tD\xi))| |D\xi|^{2} \,dt dx  \\
&\geq \Phi(z) -\lambda \nu\-int_{\B_{1}} \int_{0}^{1} (1-t)\varphi''(1+|z+tD\xi)|) |D\xi|^{2} \,dt dx  \\
&\geq \Phi(z)-\lambda \nu c \-int_{\B_{1}} \int_{0}^{1} \frac{\V'(1+|z+t D\xi|)}{1+|z+tD\xi|} |D\xi|^{2} \, dt dx  \\
&\geq \Phi(z)-\lambda \nu c \-int_{\B_{1}} \frac{\V'(1+|z|+|D\xi+z|)}{1+|z|+|D\xi+z|} |D\xi|^{2} \, dt dx \\
&\geq \Phi(z)-\lambda \nu c \-int_{\B_{1}} \frac{\V'(1+|z|+|D\xi|)}{1+|z|+|D\xi|} |D\xi|^{2} \, dt dx \\
&\geq \Phi(z) -\lambda \nu c \-int_{\B_{1}} \varphi''(1+|z|+|D\xi|) |D\xi|^{2} \, dx  \\
&\geq \Phi(z) -\lambda \nu c \-int_{\B_{1}} \varphi_{1+|z|}(|D\xi|)  \, dx. 
\end{split}
\end{equation}


\noindent
Let us take $G:=g+\Phi$ with $g$ uniformly strictly $W^{1, \varphi}$-quasiconvex with constant $k>0$ and $\Phi$ satisfying (\ref{Phi}). Consequently
\begin{align*}
\-int_{\B_{1}} G(z+D\xi) \, dx &\geq G(z) + \Bigl(k-\lambda \nu c \Bigr)\-int_{\B_{1}} \varphi_{1+|z|}(|D\xi|)\, dx\\
&= G(z) + \tilde{k}\-int_{\B_{1}} \varphi_{1+|z|}(|D\xi|)\, dx
\end{align*}
where $\tilde{k}>0$ if we choose $\nu<\frac{k}{\lambda c}$.\\
Thus $G$ is uniformly strictly $W^{1, \varphi}$-quasiconvex with constant $\tilde{k}>0$ and $G(z)=f(z)$ for $|z|>2(M+1)$. This proves the asymptotic quasiconvexity of $f$. 

\medskip
\noindent
{\it Step 2}: We want to prove that $f$ asymptotically $W^{1, \V}$- quasiconvex is equivalent to (ii), and it suffices to prove that asymptotic $W^{1, \V}$-quasiconvexity of $f$ implies (ii). Assume $f$ asymptotic $W^{1, \V}$- quasiconvex, i.e. there exist a positive constant $M$ and a uniformly strictly $W^{1,\varphi}$-quasiconvex function $g$ such that $f(z)=g(z)$ for $|z|>M$. \\
Now $g$ is locally bounded and $f$ is locally bounded from below, so we have that
$$
\alpha:=\sup_{|z|\leq M} [g(z)-f(z)]<\infty.
$$ 

\noindent
Let $R>M$ and $\eta$ be a  $C^{\infty}_{c}(\B_{R})$ function, non-negative on $\R^{Nn}$ and such that 
\begin{equation}\label{eta2}
|D^{2}\eta(z)|\leq \nu \varphi''(1+|z|) \mbox{ on }  \R^{Nn} \mbox{ and } \eta(z)\geq \alpha \mbox{ for } |z|\leq M  
\end{equation}
where $\nu$ will be chosen later.  
Let $\xi \in C^{\infty}_{c}(\B_{1})$; then we can write
\begin{equation*}
\eta(z+D\xi)=\eta(z) + (D\eta(z), D\xi) + \int_{0}^{1} (1-t)(D^{2}\eta(z+tD\xi)D\xi, D\xi) \,dt. 
\end{equation*}
Integrating over $\B_{1}$ it holds, by (\ref{eta2}),  

\begin{equation}\begin{split}\label{eta}
\-int_{\B_{1}} \eta(z+D\xi) \,dx &= \-int_{\B_{1}} \eta(z) \,dx + \-int_{\B_{1}} (D\eta(z), D\xi) \, dx  \\
&+\-int_{\B_{1}} \int_{0}^{1} (1-t)(D^{2}\eta(z+tD\xi) D\xi, D\xi) \,dt dx  \\
&\leq \eta(z) + \-int_{\B_{1}} \int_{0}^{1} (1-t) |D^{2}\eta(z+tD\xi)| |D\xi|^{2} \, dtdx  \\
&\leq \eta(z) + \nu \-int_{\B_{1}} \int_{0}^{1} \V''(1+|z+tD\xi|) |D\xi|^{2} \, dtdx  \\
&\leq \eta(z) + \nu c\-int_{\B_{1}} \varphi_{1+|z|}(|D\xi|)  \, dx.  
\end{split}
\end{equation}
where we used, as before, $\V'(t)\sim t \V''(t)$, Lemma \ref{lem2}, (\ref{simbella}) and $\V_{a}(t)\sim \V''(a+t)t^{2}$. 
 
\noindent
Now taking $G = g-\eta$, with $g$ and $\eta$ satisfying (\ref{usqc}) and (\ref{eta}), we have
\begin{equation*}
\-int_{\B_{1}} G(z+D\xi) dx \geq G(z) + \tilde{k}\-int_{\B_{1}} \varphi_{1+|z|}(|D\xi|) \, dx
\end{equation*}
where $\tilde{k}>0$ if we choose $\nu =\frac{k}{2c}$.
This means that $G$ is uniformly strictly $W^{1, \varphi}$- quasiconvex. 
But $\eta(z)\geq \alpha\geq g(z)-f(z)$ and $\eta(z)=g(z)-G(z)$, so $G(z)\leq f(z)$ for $|z|\leq M$. 

\medskip
\noindent
{\it Step 3}: The proof is similar to the previous one. 


\medskip
\noindent
{\it Step 4}: Assume that $f$ is a Borel measurable function, satisfying $(\mathcal{H}2)$. 
Since quasiconvex functions are locally Lipschitz (see \cite{FM}), we can see that $(ii)$ implies $(iv)$. 
So it suffices to show that a function satisfying $(iv)$ is asymptotically $W^{1, \varphi}$- quasiconvex. 

\noindent
Assume that $f$ satisfies $(iv)$ and consider the function
$$
F(z):=f(z)-\varepsilon \V(\sqrt{1+|z|^{2}})
$$
for $z\in \R^{Nn}$.
Here $\varepsilon>0$ will be chosen later appropriately. Now we prove that $F$ satisfies (\ref{iv1}) and (\ref{iv2}).

\noindent
Let $\xi \in C^{\infty}_{c}(\B_{1})$. Since $f$ satisfies (\ref{iv1}), we can write
\begin{equation}\begin{split}\label{A}
&\-int_{\B_{1}} F(z+D\xi) \, dx = \-int_{\B_{1}} f(z+D\xi) \, dx -\varepsilon \-int_{\B_{1}} \V(\sqrt{1+|z+D\xi|^{2}}) \, dx \\
&\geq f(z) + k\-int_{\B_{1}} \V_{|z|} (D\xi)\, dx -\varepsilon \-int_{\B_{1}} \V(\sqrt{1+|z+D\xi|^{2}}) \, dx  \\
&= F(z) + k\-int_{\B_{1}} \V_{|z|} (D\xi)\, dx -\varepsilon \-int_{\B_{1}} [\V(\sqrt{1+|z+D\xi|^{2}})-\V(\sqrt{1+|z|^{2}})] \, dx.
\end{split}
\end{equation}
Note that 
\begin{align*}
\V(\sqrt{1+|z+D\xi|^{2}}) &= \V(\sqrt{1+|z|^{2}}) + (D\V(\sqrt{1+|z|^{2}}),D\xi) \\
&+ \int_{0}^{1}\int_{0}^{1} t(D^{2}\V(\sqrt{1+|z+stD\xi|^{2}}) D\xi, D\xi) \, ds dt. 
\end{align*}
Thus integrating over $\B_{1}$ and applying Lemma \ref{keylem}, Remark \ref{REM}
 and $\V''(a+t)t^{2} \sim \V_{a}(t)$ for $a,t\geq 0$, it follows that 
\begin{equation}\begin{split}\label{B}
\-int_{\B_{1}} &[\V(\sqrt{1+|z+D\xi|^{2}})-\V(\sqrt{1+|z|^{2}})] \, dx = \-int_{\B_{1}} (D\V(\sqrt{1+|z|^{2}}),D\xi) \, dx \\
&+ \-int_{\B_{1}} \int_{0}^{1}\int_{0}^{1} t(D^{2}\V(\sqrt{1+|z+stD\xi|^{2}})D\xi, D\xi) \, ds \,dt \,dx   \\
&\leq C \-int_{\B_{1}} \int_{0}^{1}\int_{0}^{1} t\V''(\sqrt{1+|z+st D\xi|^{2}})|D\xi|^{2} \, ds \, dt \,dx  \\
&\leq C \-int_{\B_{1}} \V''(1+|z|+|D\xi|)|D\xi|^{2} \, dx  
\leq C \-int_{\B_{1}} \V_{1+|z|}(|D\xi|) \, dx  \\
&\leq C \-int_{\B_{1}} \V_{|z|}(|D\xi|) \, dx 
\end{split}
\end{equation}
for $|z|$ sufficiently large. Using together (\ref{A}) and (\ref {B}), and choosing $\varepsilon$ small enough, we have
\begin{align*}
\-int_{\B_{1}} F(z+D\xi) \, dx  \geq F(z) + K \-int_{\B_{1}} \V_{|z|}(|D\xi|) \, dx. 
\end{align*}

\noindent
Moreover, taking into account that $f$ satisfies (\ref{iv2}) we deduce, for $|z_{1}|, |z_{2}|>M$, 
\begin{align*}
|F(z_{2})-F(z_{1})|&\leq |f(z_{2}) - f(z_{1})| + \varepsilon |\V(\sqrt{1+|z_{1}|^{2}}) - \V(\sqrt{1+|z_{2}|^{2}})| \\
&\leq L |z_{2}-z_{1}| \V'(1+|z_{1}|+|z_{2}|) + \varepsilon |\V(\sqrt{1+|z_{1}|^{2}}) - \V(\sqrt{1+|z_{2}|^{2}})| \\
&\leq (L+c) |z_{2}-z_{1}| \V'(1+|z_{1}|+|z_{2}|).
\end{align*}

\noindent
Next we let 
$$
G(z):=\inf \left\{ \-int_{\B_{1}} F(z+D\xi)\, dx : \xi \in C^{\infty}_{c}(\B_{1}, \R^{N}) \right\}
$$
for $z\in \R^{Nn}$. With this definition we have that $G(z)\leq F(z)$ on $\R^{Nn}$ and $G(z)=F(z)$ for $|z|>M$. Now our aim is to prove that $G$ is locally bounded from below. 

\noindent
Fix $z\in \R^{Nn}$ such that $|z|\leq M+1$ and take $\z\in \R^{Nn}$such that $|\z|=2(M+1)$. We have
\begin{align*}
\-int_{\B_{1}} F(z+D\xi) \, dx &= \-int_{\B_{1}} [F(z+D\xi)-F(\z + D\xi)] \, dx +\-int_{\B_{1}} F(\z + D\xi) \, dx\\
&=I + II
\end{align*}
Since $F$ satisfies (\ref{iv1}) we get
$$
II=\-int_{\B_{1}} F(\z + D\xi) \, dx \geq F(\z) + k\-int_{\B_{1}} \V_{|\z|}(|D\xi|)\, dx. 
$$
Now we estimate $I$:
\begin{align*}
I&=\frac{1}{|\B_{1}|} \Bigl[ \int_{\{|D\xi|\leq 3(M+1)\}} [F(z+D\xi)-F(\z + D\xi)] \, dx \\
&+  \int_{\{|D\xi|> 3(M+1)\}} [F(z+D\xi)-F(\z + D\xi)] \, dx \Bigr]=\frac{1}{|\B_{1}|} [I_{1} + I_{2}].
\end{align*}
To estimate $I_{1}$ we use the fact that $F$ is locally bounded: $I_{1} \geq \tilde{C}$. 
Regarding $I_{2}$ we take into account that $F$ satisfies (\ref{iv2}), then we apply Young's inequality, $\V^{*}(\V'(t))\sim \V(t)$ and the $\Delta_{2}$ condition to deduce 
\begin{align*}
I_{2}&=  \int_{\{|D\xi|> 3(M+1)\}} [F(z+D\xi)-F(\z + D\xi)] \, dx \Bigr]\\
&\geq -L\int_{\{|D\xi|> 3(M+1)\}} |z-\z| \V'(1+|z+D\xi|+ |\z + D\xi|) \, dx\\
&\geq -L\delta c \int_{\{|D\xi|> 3(M+1)\}} \V(1+|z+D\xi|+|\z+D\xi|) \, dx -LC_{\delta} \int_{\{|D\xi|> 3(M+1)\}} \V(|z-\z|) \, dx\\ 
&\geq -L\delta c \int_{\{|D\xi|> 3(M+1)\}} \V(1+|\z|+|D\xi|)\, dx - C_{\delta} \\
&\geq -L\delta c \int_{\{|D\xi|> 3(M+1)\}} \V_{1+|\z|}(|D\xi|)\, dx - C_{\delta} \\
\end{align*}
where in the last inequality we used $\V_{1+|\z|}(|D\xi|)\sim \V(1+|\z| + |D\xi|)$ since $|D\xi|>1+|\z|$. 

\noindent
Putting together estimates on $I_{1}$, $I_{2}$ and $II$, taking into account that $\V_{|\z|}(t)\sim \V_{1+|\z|}(t)$ and choosing $\delta$ suitably we have
\begin{align*}
\-int_{\B_{1}} F(z+D\xi) \, dx & \geq -C\delta \int_{\{|D\xi|> 3(M+1)\}} \V_{|\z|}(|D\xi|)\, dx + F(\z) 
+ k\-int_{\B_{1}} \V_{|\z|}(|D\xi|)\, dx -C\\
&\geq -C. 
\end{align*}
So we get $G(z)\geq -C$ for $|z|\leq M+1$. Moreover for $|z|>M+1$ we gain
$$
G(z)= f(z)-\varepsilon \V(\sqrt{1+|z|^{2}}) \geq -C (1+ \V(|z|))
$$
and this proves the local boundedness of $G$ from below. \\
By Dacorogna's formula \footnote{In \cite{Dacorogna} Theorem 5 it is assumed that there exists a quasiconvex function from below $F$, and the verification of this hypothesis is not immediate in our situation. However, we may still apply the Theorem since the missing hypothesis is only needed to conclude that $G$ is locally bounded from below. Moreover, by (\ref{max}) we can say that $\V(|z|)\leq c (1+|z|^{p_{1}})$, $p_{1}>1$. } we have that $G$ coincides with the quasiconvex envelope $QF$ of $F$, and thus it is quasiconvex. 

\noindent
Finally we can prove that 
$$
g(z)= G(z) + \varepsilon \V(\sqrt{1+|z|^{2}}) \quad \mbox{for} \, z\in \R^{Nn}
$$
is a uniformly strictly $W^{1, \V}$-quasiconvex function.\\
By the quasiconvexity of $G$  we get
\begin{align*}
\-int_{\B_{1}} g(z+D\xi) \, dx &= \-int_{\B_{1}} G(z+D\xi) \, dx + \varepsilon \-int_{\B_{1}} \V(\sqrt{1+|z+D\xi|^{2}})\, dx \\
& \geq G(z) + \varepsilon \V(\sqrt{1+|z|^{2}}) + \varepsilon \-int_{\B_{1}} [ \V(\sqrt{1+|z+D\xi|^{2}})-\V(\sqrt{1+|z|^{2}})]\, dx \\
&= g(z) + \varepsilon \-int_{\B_{1}} [ \V(\sqrt{1+|z+D\xi|^{2}})-\V(\sqrt{1+|z|^{2}})]\, dx.
\end{align*}
Using Lemma \ref{keylem}, Remark \ref{REM} and $\V_{a}(t)\sim \V''(a+t)t^{2}$ it holds
\begin{align*}
&\-int_{\B_{1}} [\V(\sqrt{1+|z+D\xi|^{2}})-\V(\sqrt{1+|z|^{2}})] \, dx = \nonumber \\
&= \-int_{\B_{1}} (D\V(\sqrt{1+|z|^{2}}),D\xi) \, dx 
+ \-int_{\B_{1}} \int_{0}^{1}\int_{0}^{1} t(D^{2}\V(\sqrt{1+|z+stD\xi|^{2}})D\xi, D\xi) \, ds dt dx \nonumber \\
&= \-int_{\B_{1}} \int_{0}^{1}\int_{0}^{1} t(D^{2}\V(\sqrt{1+|z+stD\xi|^{2}})D\xi, D\xi) \, ds dt dx \\
&\geq C \-int_{\B_{1}} \int_{0}^{1}\int_{0}^{1} t \V''(\sqrt{1+|z+st D\xi|^{2}})|D\xi|^{2} \, ds dt dx \\
&\geq C \-int_{\B_{1}} \V''(1+|z|+|D\xi|)|D\xi|^{2} \, dx \\
&\geq C \-int_{\B_{1}} \V_{1+|z|}(|D\xi|) \, dx.
\end{align*}

\noindent
We deduce that $g$ is uniformly strictly $W^{1, \V}$-quasiconvex, i.e.
\begin{align*}
\-int_{\B_{1}} g(z+D\xi) \, dx \geq g(z) + \varepsilon c \-int_{\B_{1}} \V_{1+|z|}(|D\xi|) \, dx.
\end{align*}
Moreover we have that $g(z)=f(z)$ for $|z|>M+1$.
This proves that $f$ is asymptotically quasiconvex. 

\end{proof}


\section{\bf Caccioppoli estimate}
\noindent
The starting point for the investigation of the regularity properties of weak solutions is a Caccioppoli-type inequality. 

\noindent
We need the following Lemma (see Lemma 10 \cite{DLSV}):
\begin{lem}\label{iterazione}
Let $\psi$ be an $N$-function with $\psi\in \Delta_{2}$, let $r>0$ and let $h\in L^{\psi}(\B_{2r}(x_{0}))$. Further, let $f:[\frac{r}{2},r]\rightarrow [0, \infty)$ be a bounded function such that for all $\frac{r}{2} <s<t<r$
$$
f(s)\leq \theta f(t) +A \int_{\B_{t}(x_{0})} \psi \Bigl(\frac{|h(y)|}{t-s} \Bigr) \, dy ,
$$
where $A>0$ and $\theta \in [0,1)$. Then 
$$
f\Bigl(\frac{r}{2}\Bigr)\leq C(\theta, \Delta_{2}(\psi)) A \int_{\B_{r}(x_{0})} \psi \Bigl( \frac{|h(y)|}{2r} \Bigr) \, dy.  
$$
\end{lem}

\begin{thm}\label{thmcacc}
Let $u\in W^{1, \V}_{loc}(\Omega)$ be a minimizer of $\F$ and let $\B_{R}$ be a ball such that $\B_{2R}\Subset \Omega$. Then 
\begin{equation*}
\int_{\B_{R}} \V_{|z|} (|Du-z|) \, dx \leq c \int_{\B_{2R}} \V_{|z|}\Bigl( \frac{|u-q|}{R}\Bigr) dx
\end{equation*}
for all $z\in \R^{Nn}$ with $|z|>M$ and all linear polynomials $q$ on $\R^{n}$ with values in $\R^{N}$ such that $Dq=z$. 
\end{thm}

\begin{proof}
Let $0<s<t$ and consider $\B_{s}\subset \B_{t} \subset \Omega$. Let $\eta \in C^{\infty}_{c}(\B_{t})$ be a standard cut-off function between $\B_{s}$ and $\B_{t}$, such that $|D\eta| \leq \frac{c}{t-s}$.

\noindent
Define $\xi= \eta (u-q)$ and $\zeta = (1-\eta) (u-q)$; then $D\xi + D\zeta = Du-z$.

Consider 
$$
\mathcal{I} := \int_{\B_{t}} [ f(z+D\xi) - f(z) ]  \, dx. 
$$
By hypothesis $f$ is asymptotically $W^{1, \V}$-quasiconvex, and by Theorem \ref{thm1} we know that f satisfies (iv), so for $|z|>M$ we have
\begin{equation}\label{I}
\mathcal{I}\geq k \int_{\B_{t}} \V_{|z|} (|D\xi|) \, dx. 
\end{equation}
Moreover 
\begin{align*}
\mathcal{I} &= \int_{\B_{t}} [f(z+D\xi) - f(Du)+ f(Du)- f(Du-D\xi) + f(Du-D\xi) - f(z) ]\, dx\\
&=  \int_{\B_{t}} [f(z+D\xi) - f(z+D\xi +D\zeta)] \, dt + \int_{\B_{t}} [f(Du)- f(Du-D\xi)] \, dt \\
&+ \int_{\B_{t}} [f(z+D\zeta) - f(z)] \, dx=\mathcal{I}_{1}+\mathcal{I}_{2}+\mathcal{I}_{3}.
\end{align*}
Note that $\mathcal{I}_{2}\leq 0$ since $u$ is a minimizer. Let us concentrate on $\mathcal{I}_{1}$: 
\begin{align*}
\mathcal{I}_{1}=-\int_{\B_{t}} \int_{0}^{1} Df(z+D\xi + \theta D\zeta) D\zeta \,d\theta dx.
\end{align*}
Analogously concerning $\mathcal{I}_{3}$, we have
\begin{align*}
\mathcal{I}_{3}=\int_{\B_{t}} \int_{0}^{1} Df(z +\theta D\zeta)D\zeta \,d\theta \,dx.
\end{align*}
Thus we obtain that
\begin{align*}
\mathcal{I}_{1}+\mathcal{I}_{3}& 
= \int_{\B_{t}} \int_{0}^{1} [Df(z +\theta D\zeta) - Df(z+D\xi + \theta D\zeta)] D\zeta \,d\theta dx \\
& = \int_{\B_{t}} \int_{0}^{1} [Df(z +\theta D\zeta) -Df(z)+ Df(z) - Df(z+D\xi + \theta D\zeta)] D\zeta \,d\theta dx \\
&=\int_{\B_{t}} \int_{0}^{1} [Df(z +\theta D\zeta) -Df(z)] D\zeta \,d\theta dx  \\
&- \int_{\B_{t}} \int_{0}^{1} [Df(z+D\xi + \theta D\zeta)-Df(z)] D\zeta \,d\theta dx
\end{align*}
from which 
\begin{align*}
\mathcal{I}_{1}+\mathcal{I}_{3}&\leq \int_{\B_{t}} \int_{0}^{1} |Df(z +\theta D\zeta) -Df(z)| |D\zeta| \,d\theta dx \\
& + \int_{\B_{t}} \int_{0}^{1} |Df(z+D\xi + \theta D\zeta)-Df(z)| |D\zeta| \,d\theta dx. 
\end{align*}
By using hypothesis $(\mathcal{H}4)$ and Lemma \ref{lem2} we have
\begin{align*}
\int_{\B_{t}}\int_{0}^{1} &|Df(z +\theta D\zeta) -Df(z)| |D\zeta| \,d\theta dx \\
& \leq \int_{\B_{t}}\int_{0}^{1} \int_{0}^{1} |D^{2} f(tz+(1-t)(z+\theta D\zeta)|\,|\theta D\zeta|\, |D\zeta|\, dt d\theta dx \\
&\leq c \int_{\B_{t}}\int_{0}^{1} \int_{0}^{1} \V''(|tz+(1-t)(z+\theta D\zeta)|) \, |D\zeta|^{2} dt d\theta dx \\
&\leq c \int_{\B_{t}} \V''(2|z|+|z+D\zeta|) \, |D\zeta|^{2} dx  \\
&\leq c\int_{\B_{t}} \frac{\V'(2|z|+|z+D\zeta|)}{2|z|+|z+D\zeta|} |D\zeta|^{2} dx. 
\end{align*}
Taking into account the $\Delta_{2}$ condition for $\V'$ and $\V_{a}(t)\sim \V''(a+t) t^{2}$, it follows that
\begin{align*}
\int_{\B_{t}}\int_{0}^{1} |Df(z +\theta D\zeta) -Df(z)| |D\zeta| \,d\theta dx &\leq c\int_{\B_{t}} \frac{\V'(|z|+|D\zeta|)}{|z|+|D\zeta|} |D\zeta|^{2} dx \\
&\leq c\int_{\B_{t}} \V''(|z|+|D\zeta|) |D\zeta|^{2} dx\\
&\leq c\int_{\B_{t}} \V_{|z|}(|D\zeta|) dx.
\end{align*}
Analogously we can deduce
\begin{align*}
&\int_{\B_{t}} \int_{0}^{1} |Df(z+D\xi + \theta D\zeta)-Df(z)| |D\zeta| \,d\theta dx \leq \\
&\leq \int_{\B_{t}} \int_{0}^{1} \int_{0}^{1} |D^{2}f(t(z+D\xi+\theta D\zeta)+(1-t)z)| \, |D\xi + \theta D\zeta| \,|D\zeta| \, dt d\theta dx\\
&\leq c\int_{\B_{t}} \int_{0}^{1} \int_{0}^{1}  \V''(|t(z+D\xi+\theta D\zeta)+(1-t)z|) \, |D\xi + \theta D\zeta| \,|D\zeta| \, dt d\theta dx\\
&\leq c\int_{\B_{t}} \V''(|z|+|D\xi|+|D\zeta|) \, (|D\xi| + |D\zeta|) \,|D\zeta| \, dx\\
&\leq c\int_{\B_{t}} \V'_{|z|}(|D\xi|+|D\zeta|) \,|D\zeta| \, dx\\
&\leq c\int_{\B_{t}} \V'_{|z|}(|D\xi|) \,|D\zeta| \, dx +c\int_{\B_{t}} \V'_{|z|}(|D\zeta|) \,|D\zeta| \, dx \\
&\leq c\int_{\B_{t}} \V'_{|z|}(|D\xi|) \,|D\zeta| \, dx +c\int_{\B_{t}} \V_{|z|}(|D\zeta|) \, dx 
\end{align*}
where in the last line we used the equivalence $\V'_{a}(t)\sim t\V''(a+t)$ and the fact that
\begin{align*}
\V'_{|z|}(|D\xi|+|D\zeta|)
&\leq c\V'_{|z|}(|D\xi|) + c\V'_{|z|}(|D\zeta|).
\end{align*}

Applying Young's inequality for $\V_{a}$ we have
\begin{align*}
\mathcal{I}_{1}+\mathcal{I}_{3}&\leq c\int_{\B_{t}} \V_{|z|}(|D\zeta|)\, dx + c\int_{\B_{t}} \V'_{|z|}(|D\xi|) \,|D\zeta| \, dx \\
&\leq c\int_{\B_{t}} \V_{|z|}(|D\zeta|)\, dx + c\delta \int_{\B_{t}} \V_{|z|}(|D\xi|)\, dx + C_{\delta} \int_{\B_{t}} \V_{|z|}(|D\zeta|) \, dx \\
&\leq C'_{\delta} \int_{\B_{t}} \V_{|z|} (|D\zeta|) \, dx + c \delta \int_{\B_{t}} \V_{|z|} (|D\xi|) \, dx.  
\end{align*}

\noindent
Taking into account (\ref{I}) and choosing $\delta$ such that $k-c\delta>0$ we conclude
$$
\int_{\B_{t}} \V_{|z|} (|D\xi|) \, dx \leq C \int_{\B_{t}} \V_{|z|} (|D\zeta|) \, dx.
$$
Now, by the definition of $\zeta$ we have $D\zeta=(1-\eta)(Du-z)-D\eta (u-q)$ and we can note that $D\zeta=0$ in $\B_{s}$. 
Moreover using the convexity of $\V_{|z|}$ and the fact that $|D\eta| \leq \frac{c}{t-s}$, we have
\begin{align*}
\V_{|z|}(|D\zeta|) 
\leq \V_{|z|}\Bigl ((1-\eta)|Du-z|+ \frac{c}{t-s}|u-q| \Bigr)
\leq c\V_{|z|}(|Du-z|) + c \V_{|z|}\Bigl(\frac{|u-q|}{t-s} \Bigr).
\end{align*}
Hence
\begin{align*}
\int_{\B_{t}} \V_{|z|} (|D\xi|) \, dx &\leq C \int_{\B_{t}\setminus \B_{s}} \V_{|z|} (|D\zeta|) \, dx \\
&\leq c \int_{\B_{t}\setminus \B_{s}} \V_{|z|}(|Du-z|) \, dx + c \int_{\B_{t}} \V_{|z|}\Bigl(\frac{|u-q|}{t-s} \Bigr) \, dx.
\end{align*}
Thus we have
\begin{align*}
\int_{\B_{s}} \V_{|z|} (|Du-z|) \, dx & = \int_{\B_{s}} \V_{|z|} (|D\xi|) \, dx \\
& \leq \int_{\B_{t}} \V_{|z|} (|D\xi|) \, dx\\
& \leq c \int_{\B_{t}\setminus \B_{s}} \V_{|z|}(|Du-z|) \, dx + c \int_{\B_{t}} \V_{|z|}\Bigl(\frac{|u-q|}{t-s} \Bigr) \, dx.
\end{align*}
We fill the hole by adding to both sides the term $\displaystyle{c \int_{\B_{s}} \V_{|z|}(|Du-z|) \, dx}$ and we divide by $c+1$, thus obtaining 
\begin{align*}
\int_{\B_{s}} \V_{|z|} (|Du-z|) \, dx &\leq \frac{c}{c+1}  \int_{\B_{t}}  \V_{|z|} (|Du-z|) \, dx  + C \int_{\B_{t}} \V_{|z|}\Bigl(\frac{|u-q|}{t-s} \Bigr) \, dx\\
&= \lambda \int_{\B_{t}}  \V_{|z|} (|Du-z|) \, dx + \alpha \int_{\B_{t}} \V_{|z|}\Bigl(\frac{|u-q|}{t-s} \Bigr) \, dx
\end{align*}
where $\lambda:= \frac{c}{c+1}<1$ and $\alpha>0$. Now we can apply Lemma \ref{iterazione} to get the desired result. 

\end{proof}

\noindent
An immediate consequence of the previous result is the following:
\begin{cor}\label{cc1}
There exists $\alpha \in (0,1)$ such that for all minimizers $u\in W^{1, \V}(\Omega)$ of $\mathcal{F}$,
all balls $\B_{R}$ with $\B_{2R} \Subset \Omega$, and all $z\in \R^{Nn}$ with $|z|>M$
$$
\-int_{\B_{R}} |V(Du)-V(z)|^{2} dx \leq 
c \Bigl(\-int_{\B_{2R}} |V(Du)-V(z)|^{2\alpha} dx\Bigr)^{\frac{1}{\alpha}}
$$
\end{cor}
\begin{proof}
By using Lemma \ref{lem4}, applying Theorem \ref{thmcacc} with $q$ such that $(u-q)_{\B_{2R}}=0$ and Theorem \ref{SP} we have
\begin{align*}
\-int_{\B_{R}} |V(Du)-V(z)|^{2} dx &\leq c \, \-int_{\B_{R}} \V_{|z|}(|Du-z|) \, dx \\
&\leq c \-int_{\B_{2R}} \V_{|z|}\Bigl(\frac{|u-q|}{R}\Bigr)\, dx \\
&\leq c \Bigl(\-int_{\B_{2R}} \V_{|z|}^{\alpha} (|Du-z|)\, dx \Bigr)^{\frac{1}{\alpha}} \\
&\leq c\Bigl(\-int_{\B_{2R}} |V(Du)-V(z)|^{2\alpha} dx\Bigr)^{\frac{1}{\alpha}}. 
\end{align*}
\end{proof}

\noindent
Using Gehring's Lemma we deduce the following result. 
\begin{cor}\label{cc2}
There exists $s >1$ such that for all minimizers $u\in W^{1, \V}(\Omega)$ of $\mathcal{F}$,
all balls $\B_{R}$ with $\B_{2R} \Subset \Omega$, and all $z\in \R^{Nn}$ with $|z|>M$
$$
\Bigl( \-int_{\B_{R}} |V(Du)-V(z)|^{2s} dx \Bigr)^{\frac{1}{s}}\leq 
c \-int_{\B_{2R}} |V(Du)-V(z)|^{2} dx
$$
\end{cor}

\section{\bf Almost $\mathcal{A}$-harmonicity}
\noindent
In this section we recall a generalization of the $\A$-harmonic approximation Lemma in Orlicz space (see \cite{DLSV}). 

\noindent
We say that $\A= (\A_{ij}^{\alpha \beta})_{\overset{i,j=1, \cdots, N}{ \alpha, \beta=1, \cdots, n}}$ is strongly elliptic in the sense of Legendre- Hadamard if 
$$
\A(a\otimes b, a\otimes b)\geq k_{\A} |a|^{2} |b|^{2} 
$$
holds for all $a\in \R^{N}, b \in \R^{n}$ for some constant $k_{\A}>0$. 
We say that a Sobolev function $w$ on $\B_{R}$ is $\mathcal{A}$-harmonic if
$$
-\dive(\mathcal{A} Dw)=0
$$
is satisfied in the sense of distributions. \\
Given a function $u\in W^{1,2}(\B_{R})$, we want to find a function $h$ that is $\mathcal{A}$-harmonic and is close to $u$. In particular, we are looking for a function $h\in W^{1,2}(\B_{R})$ such that
\begin{equation*}
\left\{
\begin{array}{ll}
-\dive(\mathcal{A} Dh)=0 &\mbox{ in } \B_{R} \\
h=u &\mbox{ on } \partial \B_{R}
\end{array}.
\right.
\end{equation*}
Let $w:=h-u$, then $w$ satisfies 
\begin{equation}\label{aarmonica}
\left\{
\begin{array}{ll}
-\dive(\mathcal{A} Dw)=-\dive(\mathcal{A} Du) &\mbox{ in } \B_{R} \\
w=0 &\mbox{ on } \partial \B_{R}
\end{array}.
\right.
\end{equation}
We recall Theorem 14 in \cite{DLSV}:
\begin{thm}\label{14}
Let $\B_{R}\Subset \Omega$ and let $\tilde{\B}\subset \Omega$ denote either $\B_{R}$ or $\B_{2R}$. Let $\A$ be strongly elliptic in the sense of Legendre-Hadamard. Let $\psi$ be an $N$-function with $\Delta_{2}(\psi, \psi^{*})<\infty$ and let $s>1$. Then for every $\varepsilon >0$, there exists $\delta >0$ depending on $n, N, k_{\A}, |\A|,\Delta_{2}(\psi, \psi^{*}) $ and $s$ such that the following holds: let $u\in W^{1, \psi}(\tilde{\B}) $ be almost $\A$-harmonic on $\B_{R}$ in the sense that
\begin{equation*}
\Bigl|	\-int_{\B_{R}} (\A Du, D\xi) \, dx\Bigr| \leq \delta \-int_{\tilde{\B}} |Du|\, dx \|D\xi\|_{L^{\infty}(\B_{R})}
\end{equation*}
for all $\xi \in C^{\infty}_{0}(\B_{R})$. Then the unique solution $w\in W^{1, \psi}_{0}(\B_{R})$ of (\ref{aarmonica}) satisfies 
\begin{equation*}
\-int_{\B_{R}} \psi \Bigl( \frac{|w|}{R}\Bigr) \, dx + \-int_{\B_{R}} \psi(|Dw|)\, dx \leq \varepsilon \Bigl[\Bigl(\-int_{\B_{R}} \psi^{s}(|Du|)\, dx \Bigr)^{\frac{1}{s}} + \-int_{\tilde{\B}} \psi(|Du|) \, dx \Bigr]. 
\end{equation*}
\end{thm}

\noindent
The following results can be found in \cite{DLSV}.
\begin{lem}\label{dks}
Let $\B_{R}\subset \R^{n}$ be a ball and let $u\in W^{1, \V}(\B_{R})$. Then
$$
\-int_{\B_{R}} |V(Du) - (V(Du))_{\B_{R}}|^{2} dx \sim \-int_{\B_{R}} |V(Du) - V((Du)_{\B_{R}})|^{2} dx.  
$$
\end{lem}




\begin{lem}\label{thmA6}
Let $z:= (Du)_{\B_{2R}}$. For all $\varepsilon>0$ there exists $\delta>0$ such that for every $u\in W^{1, \V}(\Omega)$ minimizer of $\mathcal{F}$ and every $\B_{R}$ such that $\B_{2R}\Subset \Omega$, and for
\begin{align*}
\-int_{\B_{2R}} |V(Du)- (V(Du))_{\B_{2R}}|^{2} dx \leq \delta \-int_{\B_{2R}} |V(Du)|^{2} dx
\end{align*}
it holds
\begin{align}\label{hp5}
\Bigl| \-int_{\B_{R}} D^{2}f(z)& (Du-z, D\xi) \, dx \Bigr| \leq \varepsilon \V''(|z|) \-int_{\B_{2R}} |Du-z| \, dx \|D\xi \|_{L^{\infty}(\B_{R})},
\end{align}
for every $\xi \in C^{\infty}_{c}(\B_{R})$. 
\end{lem}


\section{\bf Excess decay estimate}

\noindent
Following the ideas in \cite{AF} we will prove the following Lemma. 

\begin{lem}\label{AF}
Let $z_{0} \in \R^{n}$ such that $|z_{0}|>1$. Let $f\in C^{2}(B_{2\sigma}(z_{0}))$ be strictly $W^{1, \V}$-quasiconvex at $z_{0}$, that is
\begin{align}\label{z0}
\int_{\B} [f(z_{0}+ D\xi) - f(z_{0})] \,dx \geq k \int_{\B} \V_{|z_{0}|}(|D\xi|) \, dx
\end{align} 
holds for all $\xi \in C^{1}_{c}(\B, \R^{N})$. Then, there exists $\rho>0$ such that for all $z\in \B_{\rho}(z_{0})$
\begin{align}\label{z1}
\int_{\B} [f(z+ D\xi) - f(z)] \,dx \geq \frac{k}{2} \int_{\B} \V_{|z_{0}|}(|D\xi|) \, dx
\end{align} 
holds for all $\xi \in C^{1}_{c}(\B, \R^{N})$. 
\end{lem}

\begin{proof}
Let 
$$
\omega_{\rho}:= \sup \{|D^{2}f(z_{1})- D^{2}f(z_{2})| : z_{1}, z_{2} \in \B_{\sigma}(z_{0}), |z_{1}-z_{2}|<\rho	\}
$$
and fix $z$ such that $|z-z_{0}|<\rho <\frac{\sigma}{2}$. \\
For $\eta \in \R^{Nn}$, define
$$
G(\eta)= f(z+\eta) - f(z_{0}+\eta). 
$$
By using (\ref{z0}) we have
\begin{align*}
&\int_{\B} [f(z+D\xi) -f(z)] \, dx =\\
& =\int_{\B} [f(z_{0}+D\xi) - f(z_{0})] \, dx +\int_{\B} [f(z+D\xi) - f(z_{0}+ D\xi) + f(z_{0}) - f(z)] \, dx \\
&\geq k \int_{\B} \V_{|z_{0}|}(|D\xi|) \, dx + \int_{\B} [G(D\xi) - G(0)- (DG(0),D\xi)]\, dx. 
\end{align*}
Now we split $\B$ as 
\begin{align*}
\X =\Bigl\{ x\in \B : |D\xi|\leq \frac{\sigma}{2}\Bigr \} \, \mbox{ and } \, \Y =\Bigl\{ x\in \B : |D\xi|> \frac{\sigma}{2}\Bigr\}. 
\end{align*}
Let us observe that 
$$
G(D\xi) - G(0)- (DG(0),D\xi) = \frac{1}{2} (D^{2}G(\theta D\xi)D\xi,D\xi)
$$ 
with $\theta \in (0,1)$. Moreover if $x\in \X$ then $|D\xi|\leq \frac{\sigma}{2}$, so $z+D\xi \in \B_{\sigma}(z_{0})$. Hence
\begin{align*}
\int_{\X} [G(D\xi) - G(0)- DG(0)D\xi] \, dx &= \frac{1}{2} \int_{\X} (D^{2} G(\theta D\xi) D\xi, D\xi) \, dx \\
&\geq -\frac{1}{2} \int_{\X} |D^{2}f(z+\theta D\xi) - D^{2}f(z_{0}+\theta D\xi)| |D\xi|^{2} dx \\
&\geq -\frac{\omega_{\rho}}{2} \int_{\X} |D\xi|^{2} dx \\
&\geq -\frac{c\, \omega_{\rho}}{2} \int_{\X} \V''(|z_{0}|+|D\xi|)|D\xi|^{2} dx \\  
&\geq -\frac{c\, \omega_{\rho}}{2} \int_{\X} \V_{|z_{0}|}(|D\xi|) dx 
\end{align*}
where we used the fact that on $\X$ we have
\begin{align*}
\V''(|z_{0}|+|D\xi|) \geq c \frac{\V'(|z_{0}|+|D\xi|)}{|z_{0}|+|D\xi|} \geq c \frac{\V'(1)}{|z_{0}|+\frac{\sigma}{2}}>0.
\end{align*}
Let us define $\displaystyle{H(z, x)= f(z + D\xi(x)) - f(z)- (Df(z), D\xi(x))}$
so that 
$$
\int_{\Y} [H(z, x) - H(z_{0}, x)] \, dx = \int_{\Y} [G(D\xi) - G(0)- (DG(0), D\xi)]\, dx.
$$
We can see that
\begin{align*}
\int_{\Y} &|H(z, x)- H(z_{0}, x)| \, dx \leq \int_{\Y} |z-z_{0}| |D_{z}H(\tau, x)| \, dx \\
&\leq \rho \left[\int_{\Y} |Df(\tau+D\xi) - Df(\tau)|\, dx	+\int_{\Y} |D^{2}f(\tau)||D\xi|\, dx \right] =\rho [\mathcal{I}+ \mathcal{II}]. 
\end{align*}
Now we estimate $\mathcal{I}$. We use hypothesis $(\mathcal{H}4)$, Lemma \ref{lem2} and the fact that $|\tau|+ |D\xi +\tau|\sim |\tau|+ |D\xi|$ to get
\begin{align*}
\mathcal{I}&\leq \int_{\Y} \int_{0}^{1} |D^{2} f(\tau +t D\xi)||D\xi| \, dt dx \\
&\leq c \int_{\Y} \int_{0}^{1} \V''(|\tau + tD\xi|) |D\xi| \, dt dx \\
&\leq c \int_{\Y} \frac{\V'(|\tau|+ |D\xi|)}{|\tau|+ |D\xi|} |D\xi| \, dx\\
&\leq c_{\sigma} \int_{\Y} \V'(|z_{0}|+|D\xi|) |D\xi|\, dx
\end{align*}
where in the last inequality we used $|\tau|+|D\xi|\leq |z|+|z_{0}|+|D\xi|\leq \rho +2|z_{0}|+|D\xi|<c(|z_{0}|+ |D\xi|)$ as well as $|\tau|+|D\xi|>1+\frac{\sigma}{2}=:c_{\sigma}$ on $\Y$, if $\rho$ is small enough.

\noindent
Analogously, we estimate $\mathcal{II}$: 
\begin{align*}
\mathcal{II}&\leq c\int_{\Y} \V''(|\tau|) |D\xi| \, dx \\
&\leq c \int_{\Y} \V'(|z_{0}|+ |D\xi| ) |D\xi| \, dx
\end{align*}
since $|\tau|\leq |z|+|z_{0}|\leq c (|z_{0}|+|D\xi|)$.\\ 
On the other hand, since on $\Y$
\begin{align*}
\V'(|z_{0}|+|D\xi|)|D\xi| &\leq c \V''(|z_{0}|+|D\xi|)(|z_{0}|+|D\xi|)|D\xi| \\
&\leq c(|z_{0}|, \sigma) \V''(|z_{0}|+ |D\xi|) |D\xi|^{2} \\
&\leq c(|z_{0}|, \sigma) \V_{|z_{0}|}(|D\xi|),
\end{align*}
we can say that 
\begin{align*}
\int_{\Y}[G(D\xi)-G(0)-(DG(0), D\xi)]\, dx \geq -\tilde{c} \rho \int_{\Y} \V_{|z_{0}|}(|D\xi|)\, dx
\end{align*}
where $\tilde{c}$ depends on the characteristics of $\V$, $\sigma$ and $|z_{0}|$.
Choosing $\rho$ such that $\frac{c\, \omega_{\rho}}{2}+ \tilde{c}\rho<\frac{k}{2}$ we have the result.

\end{proof}

\noindent
In the sequel we assume that 
$z_{0}\in \R^{n}$, with $|z_{0}|>M+1$, so that  (\ref{z1}) holds in $\B_{\rho}(z_{0})$ with $\rho <1$. 

\noindent
We define the excess function
$$
\E(\B_{R}(x_{0}),u)=\-int_{\B_{R}(x_{0})} |V(Du)-(V(Du))_{\B_{R}(x_{0})}|^{2} dx.
$$
The main ingredient to prove our regularity result is the following decay estimate: 

\begin{prop}\label{tauB}
For all $\varepsilon>0$ there exists $\delta= \delta(\varepsilon, \V)>0$ and $\beta \in (0, 1)$, such that, if $u$ is a minimizer and if for some ball $\B_{R}(x_{0})$ with $\B_{2R}(x_{0})\Subset \Omega$ the following estimates
\begin{align}\label{iii}
\E(\B_{2R}(x_{0}),u) \leq \delta \-int_{\B_{2R}(x_{0})} |V(Du)|^{2} dx,  \qquad  |(Du)_{\B_{2R}(x_{0})} - z_{0}|<\rho
\end{align}
hold true, then for every $\tau \in (0,\frac{1}{2}]$ 
\begin{equation*}
\E(\B_{\tau R}(x_{0}), u)\leq C \tau^{\beta}(\varepsilon \tau^{-n-1}+1) \E(\B_{2R}(x_{0}),u)
\end{equation*}
where $C=C(\V,n)$ and it is independent of $\varepsilon$.
\end{prop}

\begin{proof}
Let $q$ be a linear function such that $(u-q)_{\B_{2R}}=0$ and $z:=Dq = (Du)_{\B_{2R}}$. Let $w:=u-q$. Fix $\varepsilon >0$ and $\delta$ as in Lemma \ref{thmA6}, then $w$ is almost $\mathcal{A}$-harmonic with $\mathcal{A}= \frac{D^{2}f(z)}{\V''(|z|)}$. Let us observe that by Lemma \ref{AF} such $\mathcal{A}$ is strongly elliptic in the sense of Legendre-Hadamard, since for every $a\in \R^{N}$ and $b\in \R^{n}$
$$
\frac{D^{2}f(z)}{\V''(|z|)} (a\otimes b, a\otimes b) \geq \frac{\V''(|z_{0}|)}{\V''(|z|)} |a|^{2}|b|^{2}\geq c |a|^{2}|b|^{2} 
$$ 
for $z_{0}\in \R^{n}$ with $|z_{0}|>1$ and $z$ such that $|z-z_{0}|<\rho$, where c depends on $z_{0}$, $\rho$ and $\V$. 

\noindent
Let $h$ be the $\mathcal{A}$-harmonic approximation of $w$ with $h=w$ on $\partial \B_{R}$. At this point we can apply Theorem \ref{14} and conclude that, for $|z|>M$, $h$ satisfies
\begin{equation*}
\-int_{\B_{R}} \V_{|z|}(|Dw-Dh|) \,dx \leq \varepsilon \Bigl[ \Bigl(\-int_{\B_{R}} \V^{s}_{|z|}(|Du-z|)\, dx \Bigr)^{\frac{1}{s}} + \-int_{\B_{2R}}  \V_{|z|}(|Du-z|)\, dx \Bigr]
\end{equation*}
where $s$ is the same exponent of Corollary \ref{cc2}.  

\noindent
Applying Lemma \ref{lem4} and Corollary \ref{cc2} we have
\begin{align*}
\Bigl(\-int_{\B_{R}} \V^{s}_{|z|}(|Du-z|)\, dx \Bigr)^{\frac{1}{s}} &\leq c \Bigl(\-int_{\B_{R}} |V(Du)- V(z)|^{2s} dx	\Bigr)^{\frac{1}{s}} \\
&\leq c \-int_{\B_{2R}} |V(Du)- V(z)|^{2} dx
\end{align*}
from which, taking into account that $z=(Du)_{\B_{2R}}$ and using Lemma \ref{dks} we have
\begin{equation}\begin{split}\label{2b}
\-int_{\B_{R}} \V_{|z|}(|Dw-Dh|) \,dx & \leq \varepsilon c  \-int_{\B_{2R}} |V(Du)- V(z)|^{2} dx  \\
& \leq \varepsilon c  \-int_{\B_{2R}} |V(Du)- (V(Du))_{\B_{2R}}|^{2} dx  \\
& = \varepsilon c  \,\E(\B_{2R},u).
\end{split}\end{equation}
Now we want to compute $\E(\B_{\tau R},u)$. Applying Lemma \ref{dks}, Lemma \ref{lem4} and Lemma \ref{cambio} we get
\begin{align*}
\E(\B_{\tau R},u) &= \-int_{\B_{\tau R}} |V(Du) - (V(Du))_{\B_{\tau R}} |^{2} dx \\
&\leq c \-int_{\B_{\tau R}} |V(Du) - V((Dh)_{\B_{\tau R}}+z)|^{2} dx \\
&\leq c \-int_{\B_{\tau R}} \V_{|(Dh)_{\B_{\tau R}}+z|}(|Du- (Dh)_{\B_{\tau R}}-z|) \, dx \\
&= c \-int_{\B_{\tau R}} \V_{|(Dh)_{\B_{\tau R}}+z|}(|Dw- (Dh)_{ \B_{\tau R}}|) \, dx \\
&\leq C_{\eta} \-int_{\B_{\tau R}} \V_{|z|}(|Dw-(Dh)_{\B_{\tau R}}|) dx + \eta \-int_{\B_{\tau R}} |V((Dh)_{\B_{\tau R}} +z) - V(z)|^{2} dx\\
&= \mathcal{I}+\mathcal{II}.
\end{align*}

\noindent
Using Jensen's inequality, (\ref{2b}), the fact that
$$
\sup_{\B_{\tau R}} |Dh-(Dh)_{\B_{\tau R}}|\leq c \, \tau\,  \-int_{\B_{R}} |Dh- (Dh)_{\B_{R}}|\, dx
$$  
(see \cite{G}), the convexity of $\V$, and the $\Delta_{2}$-condition, we  have
\begin{align*}
\mathcal{I}&\leq C_{\eta} \-int_{\B_{\tau R}} \V_{|z|}(|Dw-Dh|) \, dx + C_{\eta} \-int_{\B_{\tau R}} \V_{|z|} (|Dh- (Dh)_{\B_{\tau R}}|)\, dx \\
& \leq C_{\eta} \tau^{-n} \varepsilon \E(\B_{2R},u) + C_{\eta} \V_{|z|}\Bigl(\tau \-int_{\B_{R}}|Dh-(Dh)_{\B_{R}}| \, dx\Bigr). 
\end{align*}
Taking into account that $\V_{a}(st)\leq c s \V_{a}(t)$ for all $a\geq 0$, $s\in [0,1]$ and $t\geq 0$, using Jensen inequality and (\ref{2b}) we have
\begin{align*}
\V_{|z|}\Bigl(&\tau \-int_{\B_{R}}|Dh-(Dh)_{\B_{R}}| \, dx\Bigr) \leq \\
&\leq c\, \tau \V_{|z|}\Bigl(\-int_{\B_{R}}|Dh-(Dh)_{\B_{R}}| \, dx\Bigr) \\
&\leq c\,\tau \V_{|z|}\Bigl(\-int_{\B_{R}}|Dh-Dw| \, dx + \-int_{\B_{R}} |Dw-(Dw)_{\B_{R}}| \, dx	\Bigr) \\
&\leq c\, \tau \V_{|z|}\Bigl(\-int_{\B_{R}}|Dh-Dw| \, dx \Bigr) + c\, \tau \V_{|z|}\Bigl(\-int_{\B_{R}}|Dw-(Dw)_{\B_{R}}| \, dx \Bigr) \\
&\leq c\, \tau \-int_{\B_{R}} \V_{|z|}(|Dh-Dw|) \, dx  +c\,  \tau \-int_{\B_{R}} \V_{|z|}(|Du-(Du)_{\B_{R}}|) \, dx  \\
&\leq c\,\tau \varepsilon \E(\B_{2R},u)  + c\, \tau \-int_{\B_{R}} \V_{|z|}(|Du-(Du)_{\B_{R}}|) \, dx  \\
&\leq c\,\tau \varepsilon \E(\B_{2R},u)  +  c\, \tau \E(\B_{2R},u)
\end{align*}
where in the last inequality we used
\begin{align*}
\-int_{\B_{R}} \V_{|z|}(|Du-(Du)_{\B_{R}}|)\, dx &\leq c\, \-int_{\B_{R}} \V_{|z|}(|Du-z|) \, dx + c\, \-int_{\B_{R}} \V_{|z|}(|z-(Du)_{\B_{R}}|)\, dx \\
&\leq c\E(\B_{2R},u) + c\, \V_{|z|} \Bigl(\Bigl|\-int_{\B_{R}} [Du-z]\, dx \Bigr|	\Bigr) \\
&\leq c \E(\B_{2R},u) + c\, \-int_{\B_{R}} \V_{|z|}(|Du-z|)\, dx\\
&\leq c\E(\B_{2R},u).
\end{align*}
So we have
\begin{align*}
\mathcal{I}
&\leq C_{\eta} \tau^{-n} \varepsilon \E(\B_{2R},u) + C_{\eta}  \tau \varepsilon \E(\B_{2R},u)  +  C_{\eta}\, \tau \E(\B_{2R},u).
\end{align*}
Now we estimate $\mathcal{II}$; taking into account that 
$$
\sup_{\B_{\tau R}} |Dh| \leq \-int_{\B_{R}} |Dh|\, dx,
$$
using Jensen's inequality, and (\ref{2b}) we obtain  
\begin{align*}
\mathcal{II}&\leq c \, \eta \, \-int_{\B_{\tau R}} \V_{|z|}(|(Dh)_{\B_{\tau R}}|)\, dx \\
&\leq c\, \eta \, \V_{|z|} \Bigl(\-int_{\B_{R}} |Dh|\, dx	\Bigr) \\
&\leq c\, \eta \, \V_{|z|} \Bigl(\-int_{\B_{R}} |Dh-Dw|\, dx +\-int_{\B_{R}} |Dw|\, dx	\Bigr)\\
&\leq c\, \eta \, \V_{|z|} \Bigl(\-int_{\B_{R}} |Dh-Dw|\, dx \Bigr) + c\, \eta \, \V_{|z|} \Bigl(\-int_{\B_{R}} |Du-z|\, dx	\Bigr)\\
&\leq c\, \eta \, \-int_{\B_{R}} \V_{|z|} (|Dh-Dw|)\, dx + c\, \eta \, \-int_{\B_{R}} \V_{|z|} (|Du-z|)\, dx\\
&\leq c\, \eta \,\varepsilon \,\E(\B_{2R},u) + c\, \eta \,\E(\B_{2R},u).
\end{align*}
Putting together estimates for $\mathcal{I}$ and $\mathcal{II}$ we have 
\begin{equation*}
\E(\B_{\tau R},u)\leq C \E(\B_{2R},u) [C_{\eta}\, \tau^{-n}\, \varepsilon +C_{\eta}\, \tau \,\varepsilon +C_{\eta}\, \tau + \eta \, \varepsilon + \eta],
\end{equation*}
choosing $\eta= \tau^{\alpha}$, and consequently $\displaystyle{C_{\eta}= \frac{1}{\tau^{\alpha (\bar{p}-1)}}}$, with $\displaystyle{\alpha <\frac{1}{\bar{p}-1}}$, we have
\begin{equation*}
\E(\B_{\tau R},u) \leq C \tau^{\beta} (\varepsilon \tau^{-n-1} +1)\E(\B_{2 R},u)
\end{equation*}
where $\beta = \min\{\alpha, 1-\alpha(\bar{p}-1) \}$. 

\end{proof}

\begin{prop}\label{prop2beta}
Let $\gamma \in (0,1)$. Then there exists $\delta$ that depends on $\gamma$ and on the characteristics of $\V$ such that: if for some ball $\B_{R}(x_{0})\subset \Omega$ 
\begin{equation}\label{eee}
\E(\B_{2R}(x_{0}),u) \leq \delta \-int_{\B_{2R}(x_{0})} |V(Du)|^{2} dx,  \qquad  |(Du)_{\B_{2R}(x_{0})} - z_{0}|<\frac{\rho}{2}
\end{equation}
hold, then for any $\rho \in (0,1]$
\begin{equation}\label{decay}
\E(\B_{\rho R}(x_{0}),u)\leq c \rho^{\gamma \beta} \E(\B_{2R}(x_{0}),u)
\end{equation}
where $c$ depends on the characteristics of $\V$. 
\end{prop}

\begin{proof}
Let $\Lambda(\varepsilon,\tau)= C\tau^{\beta}(\varepsilon \tau^{-n-1}+1)$ where $C$ depends on the characteristics of $\V$ and on $n$. 
Let $\varepsilon= \varepsilon(\tau)$ such that 
$$
\Lambda(\varepsilon, \tau)\leq \min \Bigl\{\Bigl(\frac{\tau}{2}\Bigr)^{\gamma \beta}, \frac{1}{4} \Bigr\}.
$$
Let $\delta= \delta(\tau)$ such that Proposition \ref{tauB} holds true and so small that are verified
$$
(1+\tau^{-\frac{n}{2}})\delta^{\frac{1}{2}}< \frac{1}{2} \quad \mbox{ and } \quad c\frac{\delta^\frac{1}{p}}{\tau^\frac{n}{p}}<\frac{\rho}{2},
$$
where $c$ and $p$ will be specified later. 

\noindent
With these choices we can prove that the inequalities in (\ref{iii}) hold when we replace $\B_{2R}$ with $\B_{\tau R}$, the first one being necessary to obtain the first inequality following exactly the lines of the proof of Proposition 28 in \cite{ELM2}.

\noindent
Concerning the second inequality we first observe that
$$
|(Du)_{\B_{\tau R}} -z_{0}| < |(Du)_{\B_{\tau R}}- (Du)_{\B_{2 R}}| + \frac{\rho}{2}.  
$$
Moreover, taking into account that $\V$ is of type $(p_{0}, p_{1})$ and using Lemma \ref{lem4}, for some $p>1$ we get
\begin{align*}
|(Du)_{\B_{\tau R}} - (Du)_{\B_{2 R}}|  
& \leq \-int_{\B_{\tau R}} |Du- (Du)_{\B_{2 R}}| \, dx  \\
& \leq \Bigl( \-int_{\B_{\tau R}} |Du- (Du)_{\B_{2 R}}|^{p} dx \Bigr)^\frac{1}{p} \\
& \leq c\, \Bigl( \-int_{\B_{\tau R}} \V_{|(Du)_{\B_{2R}}|}|Du- (Du)_{\B_{2 R}}| dx \Bigr)^\frac{1}{p} \\
& \leq c\, \Bigl( \-int_{\B_{\tau R}} |V(Du)- (V(Du))_{\B_{2 R}}|^2 dx \Bigr)^\frac{1}{p} \\	
& \leq \frac{c}{\tau^{\frac{n}{p}}} \Bigl( \-int_{\B_{2 R}} |V(Du)- (V(Du))_{\B_{2 R}}|^2 dx \Bigr)^\frac{1}{p} \\	
& \leq \frac{c}{\tau^{\frac{n}{p}}} \delta^\frac{1}{p} \Bigl(\-int_{\B_{2R}} |V(Du)|^{2} dx\Bigr)^{\frac{1}{p}}  \leq c\frac{\delta^\frac{1}{p}}{\tau^{\frac{n}{p}}}
\end{align*}
where in the last inequality we use that by Lemma \ref{lem4} and Jensen inequality 
\begin{align}\label{ccc}
\-int_{\B_{2R}(x)} |V(Du)|^{2} dy &\sim \-int_{\B_{2R}(x)}  \V(|Du|)\, dy \nonumber \\
&\geq \V \Bigl( \-int_{\B_{2R}(x)}  |Du|\, dy \Bigr) \geq \V(|(Du)_{\B_{2R}}|) \geq \V(M)>0. 
\end{align}
So, the smallness assumptions in (\ref{iii}) are satisfied for $\B_{\tau R}$. 
By induction we get
\begin{align*}
\E\Bigl(\B_{(\frac{\tau}{2})^{k} 2R} \Bigr)\leq \min\Bigl\{ \Bigl(\frac{\tau}{2}\Bigr)^{\gamma \beta k}, \frac{1}{4^{k}}\Bigr\} \E(\B_{2R})
\end{align*}
which is the claim. 

\end{proof}

\begin{thm}\label{C1alpha}
Let $z_{0}\in \R^{n}$ with $|z_{0}|>M+1$ so that (\ref{z1}) holds in $\B_{\rho}(x_{0})$, and let $u\in W^{1, \V}(\Omega, \R^{N})$ be a minimizer of $\mathcal{F}$. If for some $x_{0}\in \Omega$ 
\begin{align}\label{limiterho}
\lim_{r \rightarrow 0} \-int_{\B_{r}(x_{0})} |V(Du)-V(z_{0})|^{2} =0
\end{align}
then in a neighborhood of $x_{0}$ the minimizer $u$ is $C^{1, \overline{\alpha}}$ for some $\overline{\alpha}<1$.  
\end{thm}

\begin{proof}
By Jensen inequality and Lemma \ref{lem4} we have 
\begin{align*}
\V_{|z_{0}|}(|(Du)_{\B_{r}(x)} -z_{0}|) &\leq \V_{|z_{0}|} \Bigl(\-int_{\B_{r}(x)} |Du-z_{0}| \, dy	\Bigr) \\
&\leq \-int_{\B_{r}(x)} \V_{|z_{0}|}(|Du-z_{0}|) \, dy \\
&\leq c \-int_{\B_{r}(x)} |V(Du)- V(z_{0})|^{2} dy
\end{align*}
from which by (\ref{limiterho}) we can conclude that
$$
|(Du)_{\B_{2R}(x)}-z_{0}|< \rho
$$
for a suitable $R>0$. 
Moreover by Lemma \ref{lem4}, Jensen's inequality, (\ref{ccc}), 
and (\ref{limiterho}) we get
\begin{align*}
\E(\B_{2R}(x), u)\leq \-int_{\B_{2R}(x)} |V(Du)- V(z_{0})|^{2} dy \leq \delta \-int_{\B_{2R}(x)} |V(Du)|^{2} dy.
\end{align*}
Hence we have that the assumptions of Proposition \ref{prop2beta} are verified in a neighborhood of $x_{0}$, say in $\B_{s}(x_{0})$. 
Then by (\ref{decay}) we have 
$$
\E(\B_{\rho R}(x),u) \leq c \rho^{\gamma \beta} \E(\B_{2R}(x),u) \quad \forall x\in \B_{s}(x_{0})
$$ 
and by Campanato's characterization of H\"older continuity we deduce that $u\in C^{1, \overline{\alpha}}(\B_{s}(x_{0}))$ for some $\overline{\alpha}<1$. 
 
\end{proof}

\noindent
For $u\in W^{1, \V}(\Omega, \R^{N})$, we define the set of regular points $\mathcal{R}(u)$ by
$$
\mathcal{R}(u)= \{ x\in \Omega : u \mbox{ is Lipschitz near } x \}. 
$$
It follows that $\mathcal{R}(u)\subset \Omega$ is open. 
Finally we prove the following partial regularity result. 
\begin{cor}
Assume that $f$ satisfies $(\mathcal{H}1)-(\mathcal{H}5)$. Then, for every minimizer $u\in W^{1, \V}(\Omega, \R^{N})$ of $\mathcal{F}$, the regular set $\mathcal{R}(u)$ is dense in $\Omega$.  
\end{cor}

\begin{proof}
Using the characterization $(iv)$ of Theorem \ref{thm1} we can find $M>0$ such that the assumptions of Theorem \ref{C1alpha} are satisfied near every $z_{0}\in \R^{Nn}: |z_{0}|>M$.  
By Theorem \ref{C1alpha} we have that $u\in C^{1, \alpha}$ near every $x_{0}\in \Omega$ that satisfies 
$$
\lim_{r\rightarrow 0} \-int_{\B_{r}(x_{0})} |V(Du)-V(z_{0})|^{2} dx =0
$$
and these points $x_{0}$ belong to $\mathcal{R}(u)$. \\
By contradiction assume that some $x\in \Omega$ is not contained in $\overline{\mathcal{R}(u)}$; then in a neighborhood of $x$ we cannot find $x_{0}$ as before. Thus, $V(Du)$ is essentially bounded by $M$ on this neighborhood and $u$ is Lipschitz near $x$. Consequently $x\in \mathcal{R}(u)$ and we have reached the desired contradiction. 

\end{proof}

\end{document}